\documentclass[11pt,reqno]{amsart}
\usepackage{amsmath,epsfig,graphicx,color}
\usepackage{ifthen}
\usepackage{amssymb,latexsym}
\usepackage{subfigure}
\usepackage{hyperref}
\usepackage{caption}
\hypersetup{
urlcolor=black, 
  menucolor=black, 
  citecolor=black, 
  anchorcolor=black, 
  filecolor=black, 
  linkcolor=black, 
  colorlinks=true,
}

\newcommand{\wt }{\widetilde}

\newcommand{\pro}{probabilit}

\newcommand{\bfA}{{\mathbf A}}
\newcommand{\bfT}{{\mathbf T}}
\newcommand{\bfR}{{\mathbf R}}

\newcommand{\bfF}{{\mathbf F}}

\newcommand{\bfI}{{\mathbf I}}
\newcommand{\bfX}{{\mathbf X}}
\newcommand{\bfY}{{\mathbf Y}}

\newcommand{\st}{such that}

\newcommand{\std}{\stackrel{\rm d}{\rightarrow}}

\newcommand{\la}{\lambda}
\newcommand{\ds}{distribution}
\newcommand{\beao}{\begin{eqnarray*}}
\newcommand{\eeao}{\end{eqnarray*}}
\newcommand{\beam}{\begin{eqnarray}}
\newcommand{\eeam}{\end{eqnarray}}
\newcommand{\red}{\color{darkred}}
\newcommand{\blue}{\color{darkblue}}

\definecolor{darkblue}{rgb}{.1, 0.1,.8}
\definecolor{darkgreen}{rgb}{0,0.8,0.2}
\definecolor{darkred}{rgb}{.8, .1,.1}
\newcommand{\bco}{\begin{corrolary}}
\newcommand{\eco}{\end{corrolary}}

\newcommand{\E}{\mathbb{E}}
\renewcommand{\P}{\mathbb{P}}
\newcommand{\1}{\mathbf{1}}
\newcommand{\R}{\mathbb{R}}
\newcommand{\N}{\mathbb{N}}
\newcommand{\C}{\mathbb{C}}

\newcommand{\Frechet}{Fr\'{e}chet }

\DeclareMathOperator{\e}{e}

\newcommand{\X}{{\mathbf X}}
\newcommand{\Y}{{\mathbf Y}}

\newcommand{\dint}{\,\mathrm{d}}

\newcommand{\twonorm}[1]{\|#1\|_2}

\newcommand{\vep}{\varepsilon}
\newcommand{\nto}{n \to \infty}

\newcommand{\lhs}{left-hand side}
\newcommand{\rhs}{right-hand side}

\newcommand{\rv}{random variable}
\newcommand{\tr}{\operatorname{tr}}

\newcommand{\diag}{\operatorname{diag}}

\newcommand{\MP}{Mar\v cenko--Pastur }
\newcommand{\runs}{\operatorname{runs}}
\newcommand{\length}{\operatorname{length}}

\newcommand{\clt}{central limit theorem}

\def\tag{\refstepcounter{equation}\leqno }

\newtheorem{lemma}{Lemma}[section]

\newtheorem{theorem}[lemma]{Theorem}
\newtheorem{proposition}[lemma]{Proposition}
\newtheorem{definition}[lemma]{Definition}
\newtheorem{corollary}[lemma]{Corollary}
\newtheorem{example}[lemma]{Example}

\newtheorem{remark}[lemma]{Remark}

\newcommand{\cid}{\stackrel{d}{\rightarrow}}
\newcommand{\cip}{\stackrel{\P}{\rightarrow}}

\newcommand{\cas}{\stackrel{\rm a.s.}{\rightarrow}}

\newcommand{\eid}{\stackrel{d}{=}}

\newcommand{\var}{{\rm var}}

\newcommand{\cov}{{\rm cov}}

\newcommand{\as}{{\rm a.s.}}

\newcommand{\con}{convergence}
\newcommand{\seq}{sequence}

\newcommand{\asy}{asymptotic}
\begin{document}
\bibliographystyle{acm}
\title[Convergence of the extreme eigenvalues of the sample correlation matrix]
{Almost sure convergence of the largest and smallest eigenvalues of high-dimensional sample correlation matrices}
\thanks{Johannes Heiny's and Thomas Mikosch's research is partly supported by the Danish Research Council Grant DFF-4002-00435 ``Large random matrices with heavy tails and dependence''. Johannes Heiny thanks the Villum Foundation for the grant 11745 ``Ambit fields: Probabilistic properties and statistical inference''.}
\author[Johannes Heiny]{Johannes Heiny}
\author[Thomas Mikosch]{Thomas Mikosch}
\address{Department of Mathematics,
Ruhr University Bochum,
Universit\"atsstrasse 150,
D-44801 Bochum,
Germany}
\email{johannes.heiny@rub.de}
\address{Department  of Mathematical Sciences,
University of Copenhagen,
Universitetsparken 5,
DK-2100 Copenhagen,
Denmark}
\email{mikosch@math.ku.dk\,, www.math.ku.dk/$\sim$mikosch}
\begin{abstract}{In this paper, we show that the largest and smallest eigenvalues 
of a sample correlation matrix stemming from $n$ independent observations of a $p$-dimensional 
time series with iid components converge almost surely to $(1+\sqrt{\gamma})^2$ and $(1-\sqrt{\gamma})^2$, 
respectively, as $\nto$, if $p/n\to \gamma \in (0,1]$ and the truncated variance of the entry 
distribution is ``almost slowly varying'', a condition we describe via moment 
properties of self-normalized sums. Moreover, the empirical spectral distributions 
of these sample correlation matrices converge weakly, with probability $1$, to the \MP law, 
which extends a result in \cite{bai:zhou:2008}. We compare the behavior of the eigenvalues of the 
sample covariance and sample correlation matrices and argue that the latter seems more robust, 
in particular in the case of infinite fourth moment. We briefly address some practical issues 
for the estimation of extreme eigenvalues in a simulation study.   
\par
In our proofs we use the method of moments combined with a Path-Shortening Algorithm, 
which efficiently uses the structure of sample correlation matrices, 
to calculate precise bounds for matrix norms. We believe that this new approach could be of further use in random matrix theory. }
\end{abstract}
\keywords{Sample correlation matrix, infinite fourth moment, largest eigenvalue, smallest eigenvalue, 
spectral distribution, sample covariance matrix, self-normalization, regular variation, combinatorics.}
\subjclass{Primary 60B20; Secondary 60F05 60F10 60G10 60G55 60G70}

\maketitle
\section{Introduction and notation}\label{sec:sdlgj}

In modern statistical analyses one is often faced with large data sets where both the dimension of the observations and the sample size are large. The dramatic increase and improvement of computing power and data collection devices have triggered the necessity to study and interpret the sometimes overwhelming amounts of data in an efficient and tractable way. Huge data sets arise naturally in wireless communication, finance, natural sciences and genetic engineering.
For such data one commonly studies the dependence structure via covariances and correlations which can be estimated by their sample analogs. {\em Principal component analysis}, for example, uses an orthogonal transformation of the data such that only a few of the resulting vectors explain most of the variation in the data. The empirical variances of these so-called {\em principal component vectors} are the largest eigenvalues of the {\em sample covariance or correlation matrix}.  

Throughout this paper we consider the $p\times n$ data matrix
\begin{equation*}
\X=\X_n=(X_{it})_{i=1,\ldots,p; t=1,\ldots,n}
\end{equation*}
of identically distributed entries $(X_{it})$ with generic element $X$, where we assume $\E[X]=0$ and $\E[X^2]=1$ if the first and second moments of $X$ are finite, respectively. A column of $\X$ represents an observation of a $p$-dimensional time series. 

Random matrix theory provides a great variety of results on the ordered eigenvalues
\begin{equation}\label{eq:deflambda}
\la_{(1)} \ge \cdots \ge\la_{(p)}\,,
\end{equation}
of the (non-normalized) {\em sample covariance matrix} $\X\X'$. Here we will only discuss the case $p=p_n \to \infty$ and, unless stated otherwise, we assume the growth condition  
\begin{equation}\label{Cpn}
\lim_{\nto} \frac{p_n}{n} = \gamma\in (0,1]\,. \tag{$G_{\gamma}$}
\end{equation} 
For the finite $p$ case, we refer to \cite{anderson:1963,muirhead,janssen:mikosch:rezapour:xie:2016}.
When studying the asymptotic properties of estimators under \eqref{Cpn}
one often obtains results that dramatically differ from the standard $p$ fixed, $\nto$ case, in which the spectrum of $(n^{-1}\X\X')$ converges to its population covariance spectrum. In 1967, Mar\v cenko and Pastur \cite{marchenko:pastur:1967} observed that even in the case of iid entries $(X_{it})$ with $\E[X^2]=1$ the eigenvalues $(\lambda_{(i)}/n)$ do not concentrate around $1$. For more examples, see \cite[Chapter 1]{bai:silverstein:2010} and \cite{elkaroui:2009}. Typical applications where \eqref{Cpn} seems reasonable are discussed in \cite{johnstone:2001,donoho:2000}. 
\par

In comparison with $(\lambda_{(i)})$, much less is known about the ordered eigenvalues
\begin{equation*}
\mu_{(1)} \ge \cdots \ge\mu_{(p)}\,
\end{equation*} 
of the {\em sample correlation matrix} $\bfR=\Y\Y'$ with entries
\begin{equation}\label{eq:corrR}
R_{ij}=\sum_{t=1}^n \frac{X_{it}X_{jt}}{\sqrt{D_i} \sqrt{D_j}} = \sum_{t=1}^n Y_{it}Y_{jt}\,, \quad i,j=1,\ldots,p\,.
\end{equation}
In this paper we will often make use of the notation $\Y=(Y_{it})=(X_{it}/\sqrt{D_i})$ and
\begin{equation}\label{eq:D}
D_i=D_i^{(n)}=\sum_{t=1}^n X_{it}^2\,, \qquad
i=1,\ldots,p;\;  n\ge 1\,.
\end{equation} 
Note that the dependence of $(\lambda_{(i)})$ and $(\mu_{(i)})$ on $n$ is suppressed in the notation.

\subsection{The case $(X_{it})$ iid, $\E[X^4]<\infty$ and $\E[X^2]=1$}  

In this case the behaviors of the eigenvalues of the sample covariance matrix $\X\X'$ and the sample correlation matrix $\bfR$ are closely intertwined. 
\par 

For any random $p\times p$ matrix $\bfA$ with real eigenvalues $\lambda_1(\bfA),\ldots,\lambda_p(\bfA)$ the {\em empirical spectral distribution} is defined by 
\begin{equation*}
F_{\bfA}(x)= \frac{1}{p}\; \sum_{i=1}^p \1_{\{ \lambda_i(\bfA)\le x \}}, \qquad x\in  \R\,.
\end{equation*}
Many functionals of the eigenvalues $\lambda_1(\bfA),\ldots,\lambda_p(\bfA)$ can be expressed in terms of $F_{\bfA}$ \cite{bai:fang:liang:2014}, for instance
\begin{equation*}
\det \bfA = \prod_{i=1}^p \lambda_i(\bfA) = \exp \Big( p \int_{0}^{\infty} \log x \;\dint F_{\bfA}(x) \Big)\,.
\end{equation*}
A major problem in random matrix theory is to find the weak limit of $(F_{\bfA_n})$ for suitable sequences $(\bfA_n)$; see for example \cite{bai:silverstein:2010,yao:zheng:bai:2015} for more details. By weak convergence of a sequence of probability distributions $(F_{\bfA_n})$ 
to a \pro y \ds\ $F$, we mean $\lim_{\nto} F_{\bfA_n}(x)=F(x) \,\as$ for all continuity points of $F$.
In this context a useful tool is the {\em Stieltjes transform}
of the empirical spectral distribution $F_{\bfA}$:
\begin{equation*}
s_{\bfA}(z)= \int_{\R} \frac{1}{x-z} \dint F_{\bfA}(x) = \frac{1}{p} \tr(\bfA -z \bfI)^{-1}\,, \quad z\in \C^+\,,
\end{equation*}
where $\C^+$ denotes the complex numbers with positive imaginary part. 
Weak convergence of $(F_{\bfA_n})$ to $F$ is equivalent to $s_{F_{\bfA_n}}(z) \to s_F(z)$ a.s. for all $z\in \C^+$.

Under the growth condition \eqref{Cpn}, the sequence of empirical spectral distributions of the normalized sample covariance matrix $n^{-1} \X\X'$ converges weakly to the \MP law with density
\begin{eqnarray}\label{eq:MP}
f_\gamma(x) = 
\left\{\begin{array}{cc}
\frac{1}{2\pi x\gamma} \sqrt{(b-x)(x-a)} \,, & \mbox{if } a\le x \le b, \\
 0 \,, & \mbox{otherwise,}
\end{array}\right.
\end{eqnarray}\noindent
where $\gamma\in (0,1]$, $a=(1 -\sqrt{\gamma})^2$ and $b=(1 +\sqrt{\gamma})^2$. This classical result is sometimes referred to as \MP theorem \cite{marchenko:pastur:1967}. Informally, the histogram of $(\lambda_{(i)}/n)$ is asymptotically non-random and the limiting shape depends only on the fraction $p/n$. For an illustration, see Figure~\ref{fig:alpha6}.
\par

The \MP law has $k$-th moment
\begin{equation}\label{eq:momentsmp}
\beta_k=\beta_k(\gamma)=\int_a^b x^k f_\gamma(x) \dint x=\sum_{r=1}^{k} \frac{1}{r} \binom{k}{r-1}\binom{k-1}{r-1}\gamma^{r-1}\,,\quad k\ge 1\,,
\end{equation}
and Stieltjes transform 
\beam\label{eq:stieltjestransform}
s(z)
=\int_{\R} \frac{1}{x-z}f_\gamma(x) \dint x 
= \frac{1-\gamma -z +\sqrt{(1+\gamma-z)^2-4\gamma}}{2 \gamma z} \,;
\eeam
see \cite[Chapter~3]{bai:silverstein:2010} or \cite{bai:fang:liang:2014,yao:zheng:bai:2015}.
\par

The \as~behavior of the extreme eigenvalues is more involved and therefore it has received significant attention in the literature. From the \MP theorem one can infer
\begin{equation}\label{eq:efsdf}
\limsup_{\nto} n^{-1}\lambda_{(p)} \le (1-\sqrt{\gamma})^2 \le (1+\sqrt{\gamma})^2 \le \liminf_{\nto} n^{-1}\lambda_{(1)}\, \quad \as
\end{equation}
The finiteness of the fourth moment of $X$ is necessary for the almost sure convergence of  $\lambda_{(1)}/n$; see \cite{baisilv}.
If $\E[X^4]<\infty$, one has (see \cite{bai:silverstein:2010}) 
\begin{equation}\label{eq:drtgdrghdr}
n^{-1} \la_{(1)} \to (1+\sqrt{\gamma})^2\, \quad \text{ and } \quad n^{-1} \la_{(p)} \to (1-\sqrt{\gamma})^2\, \quad \as
\end{equation}
\par

The minimal moment requirement for the convergence of the normalized smallest eigenvalue, however, was an open question for a long time.
Recently, it was proved in \cite{tikhomirov:2015} that $n^{-1} \la_{(p)} \to (1-\sqrt{\gamma})^2\, \as$ only requires a finite second moment.
Under suitable moment assumptions $\la_{(1)}$ and $\la_{(p)}$ possess {\em Tracy--Widom} fluctuations around their almost sure limits. For instance, 
the paper \cite{johnstone:2001} complemented \eqref{eq:drtgdrghdr} by the corresponding \clt\ in the special case of iid standard normal 
entries:
\beam\label{eq:tc}
 n^{2/3}\,\dfrac{(\sqrt{\gamma})^{1/3}}{\big(1+\sqrt{\gamma}\big)^{4/3}}\Big(\dfrac {\la_{(1)}}{n} -
\big(1+\sqrt{\tfrac pn }\big)^2\Big)
\std \xi\,,
\eeam
where the limiting \rv\ has a {\em Tracy--Widom \ds} of order 1. Notice that the centering
$\big(1+\sqrt{\tfrac pn }\big)^2$ can in general not be replaced by $(1+\sqrt{\gamma})^2$.
This \ds\ is ubiquitous in random matrix theory.
Its distribution function $F_1$ is given by
\begin{equation*}
F_1(s) = \exp\Big\{
  -\frac{1}{2} \int_{s}^\infty [
    q(x) + (x - s) q^2(x)
 ] \dint x
\Big\}\,,
\end{equation*}
where $q(x)$ is the unique solution to the Painlev\'e II differential
equation
\begin{equation*}
  q''(x) = xq(x) + 2 q^3(x)\,,
\end{equation*}
where $ q(x)\sim {\rm Ai}(x)$ as $x \to \infty$ and Ai$(\cdot)$ is the Airy kernel; see Tracy and Widom~\cite{tracy:widom:2012} for details.
\par

Sometimes practitioners would like to know ``to which extent the random matrix results would hold if one were concerned with sample correlation matrices and not sample covariance matrices \cite{elkaroui:2009}''.
A partial answer is that the aforementioned results also hold for the sample correlation matrix $\bfR$ and its eigenvalues $\mu_{(1)}\ge\cdots \ge \mu_{(p)}$. 
With $\bfF= \diag (1/D_1,\ldots,1/D_p)$, we have $\bfR={\bfF}^{1/2} {\X}{\X}'{\bfF}^{1/2}$ which has the same eigenvalues as ${\X}{\X}'{\bfF}$. Weyl's inequality (see \cite{bhatia:1997}) yields
\begin{equation}\label{eq:lamu}
\begin{split}
\max_{i=1,\ldots,p} |\mu_{(i)}- n^{-1} \la_{(i)}| &\le \twonorm{{\X}{\X}'{\bfF}-n^{-1}{\X}{\X}' }\\
&\le n^{-1} \twonorm{{\X}{\X}'}  \twonorm{n \bfF- \bfI}\\
&= n^{-1} \la_{(1)} \max_{i=1,\ldots,p} \Big| \frac{n}{D_i} -1 \Big|\,,
\end{split}
\end{equation}
where for any matrix $\bfA$, $\twonorm{\bfA}$ denotes its spectral norm, i.e., its largest singular value.

Lemma 2 in \cite{bai:yin:1993} implies that $\E[X^4]<\infty$ is equivalent to
\begin{equation*}
\max_{i=1,\ldots,p} \Big| \frac{n}{D_i} -1 \Big| \cas 0\,,
\end{equation*}
while $n^{-1} \la_{(1)}\to (1+\sqrt{\gamma})^2~ \as$ Hence, $\max_{i=1,\ldots,p} |\mu_{(i)}- n^{-1} \la_{(i)}| \to 0~ \as$
This approach was used in \cite{jiang:2004,xiao:zhou:2010} to derive 
\begin{equation}\label{eq:drtgdrghdr1}
\mu_{(1)} \to (1+\sqrt{\gamma})^2\, \quad \text{ and } \quad  \mu_{(p)} \to (1-\sqrt{\gamma})^2\, \quad \as
\end{equation}
If the assumption $\E[X^4]<\infty$ is weakened to $\lim_{\nto} n\,\P(X^4> n)=0$, the paper \cite{bai:yin:1993} proves that $n^{-1} \la_{(1)}\cip (1+\sqrt{\gamma})^2$ and $\max_{i=1,\ldots,p} \big| n/D_i -1 \big| \cip 0$. As a consequence, the limit results for $\mu_{(1)}$ and $\mu_{(p)}$ hold in probability instead of \as
\par
Distributional limit results have been derived for the appropriately centered and normalized eigenvalues of sample correlation matrices. The authors of \cite{bao:pan:zhou:2012} assumed iid, symmetric entries $X_{it}$ and that there exist positive constants $C,C'$ such that $\P(|X|\ge t^C)\le \e^{-t}, t\ge C'$. They showed \eqref{eq:tc} with $\lambda_{(1)}/n$ replaced by $\mu_{(1)}$.
A similar limit result holds for $\mu_{(p)}$.

\subsection{The case $(X_{it})$ iid and $\E[X^4]=\infty$}

Asymptotic theory for the eigenvalues of $\X\X'$ in the case of an entry distribution with infinite fourth moment was studied in \cite{soshnikov:2004,soshnikov:2006,auffinger:arous:peche:2009} in the cases when $p/n\to \gamma \in (0,\infty)$, while the authors of \cite{davis:mikosch:heiny:xie:2015,heiny:mikosch:2015:iid} allowed nearly arbitrary growth of the dimension $p$. 
In their model, the entries of $\X$ are regularly varying with index $\alpha>0$, implying that
\begin{equation}\label{eq:reg}
\P(|X|>x) = x^{-\alpha} L(x)\,,
\end{equation}
for a slowly varying function $L$. For $\alpha \in (0,4)$,  which implies an infinite fourth moment, they showed that $(a_{np}^{-2} \la_{(1)})$ converges to a \Frechet distributed random variable $\eta_{\alpha/2}$ with parameter $\alpha/2$ while $a_{np}^{-2} \la_{(p)}\cip 0$. Here the normalizing sequence $(a_n)$  is defined via $\P(|X|>a_n)\sim n^{-1}$, hence $n/a_{np}^{2}\to 0$. 

To illustrate the stark contrast between the cases $\alpha>4$ and $\alpha<4$, assume \eqref{Cpn} and $\E[X]=0$ if $\E[|X|]<\infty$. Then it follows from \eqref{eq:drtgdrghdr} that
\begin{equation*}
\begin{split}
\frac{\la_{(p)}}{\la_{(1)}} &\cas \frac{(1-\sqrt{\gamma})^2}{(1+\sqrt{\gamma})^2} \quad \text{ if } \alpha >4\,,\\
\frac{a_{np}^2}{n} \frac{\la_{(p)}}{\la_{(1)}} &\cid \frac{(1-\sqrt{\gamma})^2}{\eta_{\alpha/2}} \quad \text{ if } \alpha \in (2,4)\,,\\
\frac{a_{np}^2}{n} \frac{\la_{(p)}}{\la_{(1)}} &\cas 0 \quad \text{ if } \alpha \in (0,2)\,,
\end{split}
\end{equation*}
where the rate $a_{np}^2/n\to \infty$ in the last line can even be increased. 
To the best of our knowledge, a suitable normalization $(b_n)$ such that $(b_n \lambda_{(p)})$ has a nontrivial limit is not available when $ \alpha \in (0,2)$. 

Under \eqref{Cpn} the asymptotic behavior of the eigenvalues of sample correlation matrices can be very different from that of sample covariance matrices, especially for an entry distribution with infinite fourth moment. If $\alpha\in (2,4)$, the \MP theorem and Theorem 2.3 in \cite{bai:zhou:2008} assert that $(F_{n^{-1}\X\X'})$ and $(F_{\bfR})$  converge weakly to the \MP law. From \cite{baisilv} it is known that $\limsup_n \lambda_{(1)}/n =\infty \,\as$ 
\par

For $\E[X^4]=\infty$, the approach to sample correlation matrices from \eqref{eq:lamu} fails. No limit results for $\mu_{(1)}$ or $\mu_{(p)}$ seem to be available in the literature at this point, although Theorem 2.3 in \cite{bai:zhou:2008} ensures the weak convergence of the empirical spectral distribution $F_{\bfR}$ to the \MP law if $X$ is in the domain of attraction of the normal distribution. Analogously to \eqref{eq:efsdf}, the weak limit of $(F_{\bfR})$ provides a first idea what the limits of the extreme eigenvalues might be. 

\subsection{$(X_{it})$ identically distributed, but dependent}
For practical purposes it is important to work with arbitrary population covariance matrices and not just $n^{-1} \E[\X\X']=\bfI$. Based on well understood results in the iid case, numerous generalizations and estimation techniques have been developed. For many models the limiting spectral distribution can only be characterized in terms of an integral equation (=\MP equation) for its Stieltjes transform. Explicit solutions are more involved; see the monographs \cite{bai:silverstein:2010,bai:fang:liang:2014,yao:zheng:bai:2015}. Over the last couple of years significant progress on limiting spectral distributions for dependent time series was achieved; see for example 
\cite{banna:merlevede:peligrad:2015,banna:merlevede:2015,banna:2016}.
Since the sample covariance matrix is a poor estimator for the population covariance matrix in high dimension, a different approach to the fundamental problem of estimating population eigenvalues is needed. In \cite{elkaroui:purdom:2016} the authors find that the bootstrap works for the top eigenvalues if they are sufficiently separated from the bulk. Among others, El Karoui \cite{elkaroui:2008} proposed to use the \MP equation, which basically requires more insight into the empirical spectral distribution and its support. This was achieved in \cite{dobriban:2015}, where an algorithm for calculating the spectral distribution based on certain approximate integral equations for its Stieltjes transform was presented.

In view of \cite{davis:pfaffel:stelzer:2014,davis:mikosch:pfaffel:2015,davis:mikosch:heiny:xie:2015} the behavior of the top eigenvalues is reasonably well understood in the case of linear dependence among the $X_{it}$ and $\E[X^4]=\infty$. 
If $\E[X^4]<\infty$, similar arguments to \eqref{eq:lamu} can be developed to show that methods for sample covariance matrices can be applied to sample correlation matrices; see for example \cite{elkaroui:2009}. Theorem 1 in \cite{elkaroui:2009} proves that if the spectral norm of the population correlation matrix is uniformly bounded and $\E[X^4 (\log X)^{2+\vep}]<\infty$, then the spectral properties of $\bfR$ and $n^{-1} \X\X'$ are asymptotically the same. In particular, if $\lambda_{(1)}/n\cas c$, then $\mu_{(1)}\cas c$.

For the sake of completeness we mention that the study of non-asymptotic high-dimensional sample covariance matrices was subject to an intense line of research in the last years. Good references are \cite{srivastava:vershynin:2013,adamczak:litvak:pajor:jaegermann:2010,adamczak:litvak:pajor:jaegermann:2011,yao:zheng:bai:2015}.

\subsection{About this paper}
In Section~\ref{sec:2} we introduce the basic assumptions of this paper and discuss their meaning.
The main results are given in Section~\ref{sec:3}.
We show that the limiting spectral distribution of the 
sample correlation matrices is the \MP law (Theorem~\ref{thm:mpcorrelation}) 
and that the extreme eigenvalues converge  \as~to the endpoints of the limiting support (Theorem~\ref{thm:mu1convergence}) provided $\bfX$ has iid entries \st\
their truncated variance is ``almost slowly varying''. 
In this sense, the limiting spectral distribution of sample correlation matrices is universal. 
A similar kind of universality holds for the limiting spectral distribution of sample covariance matrices given a finite variance, while the asymptotic behavior of their extreme eigenvalues is totally different if the fourth moment is infinite.
Thus the eigenvalues of sample correlation matrices exhibit a  ``more robust'' behavior than their sample covariance analogs. 
This is perhaps not surprising in view of the {\em self-normalizing property} of sample correlations.
Self-normalization also has the advantage that one does not have to worry about the correct normalization.
This is a crucial problem in the study of sample covariance matrices in the case of an infinite fourth moment 
where one needs a normalization stronger than the classical one.
We conclude Section~\ref{sec:3} with a small simulation study which shows that the \asy\ results work nicely.
\par
We continue with some technical results in Section~\ref{sec:5.3}. These are of independent interest
because they provide a {\em Path-Shortening Algorithm} for the calculation of bounds for the very high moments of 
$\mu_{(1)}$. We believe that this technique is novel and will be of further use for proving results
in random matrix theory. The proofs of our main results Theorems~\ref{thm:mu1convergence} and \ref{thm:mpcorrelation}
are given in Sections~ \ref{sec:5.2} and \ref{sec:6}, respectively. Both proofs heavily depend on the techniques 
developed in   Section~\ref{sec:5.3}. We conclude with an Appendix which contains some auxiliary analytical results.
\par
Condition \eqref{eq:assumptionq} is crucial for the proof of Theorem~\ref{thm:mpcorrelation}. In Section~\ref{sec:2}
we discuss this condition and find out that it is very close to condition \eqref{eq:condX}
which in turn is very close (but not equivalent) to membership of the \ds\ of $X$ in the domain of attraction
of the Gaussian law. We conjecture that the statement of Theorem~\ref{thm:mpcorrelation} may be proved only under \eqref{eq:condX}.

\section{Assumptions}\label{sec:2}\setcounter{equation}{0}
In this section we will present some distributional assumptions and discuss their meaning.
We assume that $(X_{it})$ is an iid field with generic element $X$. In order to exclude the degenerate case we assume $\var(X)>0$. Recall the notation
\begin{equation}\label{eq:Yit}
Y_{it}=\frac{X_{it}}{\sqrt{D_i}}\,, \quad i=1,\ldots,p;\, t=1,\ldots,n\,.
\end{equation}
For ease of notation we will sometimes write $(Y_1, \ldots, Y_n)= (Y_{11}, \ldots, Y_{1n})$, $Y=Y_1$ and $D=D_1$.
\subsection{Domain of attraction type-condition for the Gaussian law}
One of the basic assumptions in this paper is
\beam\label{eq:condX}
\E \big[ Y_{1}Y_{2} \big] = o(n^{-2}) \quad  \text{ and } \quad  \E \big[ Y_{1}^4 \big] = o(n^{-1})\,,\qquad \nto\,.
\eeam
In \cite{gine:goetze:mason:1997} it was proved that condition \eqref{eq:condX} holds if the \ds\ of 
$X$ is in the domain of attraction of the normal law, which is equivalent to $\E[X^2 \1_{\{|X|\le x\}}]$ being slowly varying. 
\par
The converse implication is not valid. Indeed, let $h(\cdot)$ be a positive function such that $0<c_1 = \liminf_{x \to \infty}h(x) <\limsup_{x \to \infty}h(x) =c_2<\infty$ and consider a symmetric random variable $X$ with tail $\P(X>x)=\P(X<-x)=x^{-2} h(|x|)/2$ for $x$ sufficiently large. Then we have
\begin{equation*}
c_1 = \liminf_{x \to \infty} \frac{\P(|X|>x)}{x^2} <\limsup_{x \to \infty}\frac{\P(|X|>x)}{x^2} =c_2\,,
\end{equation*} 
and therefore $\E[X^2 \1_{\{|X|\le x\}}]$ is not slowly varying, or, equivalently, the \ds\ of $X$ is not in the domain of attraction of the normal law, but \eqref{eq:condX} is valid as a domination argument shows.

Jonsson~\cite{jonsson:2010} proved that $\E[(Y_1 + \cdots + Y_n)^2]\ge 1$, with equality if and only if $X$ is symmetric.
He also gave an explicit expression for the mixed moment
\begin{equation*}
\E[Y_1 Y_2]=  \int_0^\infty (\E[X\e^{-\la X^2}])^2 (\E[\e^{-\la X^2}])^{n-2} \dint \la\,, \quad n\ge 2\,. 
\end{equation*}
In view of the identity $\E[(Y_1 + \cdots + Y_n)^2]= 1+n(n-1)\, \E[Y_1 Y_2]$, one makes the interesting observation that $\E[Y_1 Y_2]$ is always nonnegative.


\subsection{Condition~\eqref{eq:assumptionq}}\label{sec:cond}
This condition will be crucial for the proofs in this paper:\\[1mm]
{\em There exists a sequence $q=q_n\to \infty$ such that for some integer sequence $k=k_n$ 
with $k/\log n \to \infty$ we have $(k^3 q)/n \to 0$, and the moment inequality
\begin{equation}\label{eq:assumptionq}
\E[ Y_{1}^{2m_1}\cdots Y_{r}^{2m_r} ] \le 
 \frac{q_n}{n} \, \E[ Y_{1}^{2m_1}\cdots Y_{r-1}^{2m_{(r-1)}}Y_{r}^{(2m_r-2)} ]\, \tag{$C_q$}
\end{equation}
holds for $1\le r\le \ell-1$ and any positive integers $m_1,\ldots,m_r$ satisfying $m_1+\cdots +m_r=\ell$, where $\ell \le k$.}\\[1mm]
Next we shed some light on this condition. It turns out to be closely related to \eqref{eq:condX}. Indeed, assume \eqref{eq:assumptionq}. Iteration of \eqref{eq:assumptionq} for any fixed $\ell$ yields
\beao
\E[ Y_{1}^{2m_1}\cdots Y_{r}^{2m_r} ] \le  \big(\frac{q_n}{n}\big)^{\ell-r} \E[ Y_{1}^{2}\cdots Y_{r}^{2} ]
\sim \frac{q_n^{\ell-r}}{n^{\ell}} , \quad \nto\,.
\eeao
In particular, $n\,\E[Y_{1}^4]\le q_n/n\le (\log n)^{-3}$. Thus,  \eqref{eq:assumptionq} provides some precise rate at which 
$n\,\E[Y_{1}^4]$ converges to zero.
 Condition \eqref{eq:assumptionq} implies that 
\begin{equation*}
\lim_{\nto} (\log n)^3 \,n\,\E \big[ Y_{1}^4 \big] = 0\,,
\end{equation*}
which is satisfied for regularly varying distributions with index $\alpha>2$. 
\par
Moreover, \eqref{eq:assumptionq} does not hold if $\vep=\liminf_{\nto} n\,\E [ Y_{1}^4 ] >0$.
If \eqref{eq:assumptionq} were valid we would have for large $n$,
\begin{equation*}
\vep/2 \le n\,\E[Y_1^4] \le \frac{q_n}{n-1} \, n(n-1) \E[Y_1^2Y_2^2]\le \frac{q_n}{n-1}\to 0\,.
\end{equation*}
For example, Proposition~1 in \cite{mason:zinn:2005} asserts that the \ds\ of $X^2$ is in the 
domain of attraction of an $\alpha/2$-stable distribution with $0< \alpha <2$ if and only if
\begin{equation}\label{eq:limita<2}
\lim_{\nto} n\,\E[Y_{1}^4] =1-\frac{\alpha}{2}\,,
\end{equation}
hence \eqref{eq:assumptionq} does not hold if $|X|$ has a regularly varying tail with index $0<\alpha<2$.
\par

The expectations in \eqref{eq:assumptionq} can be calculated by using the following formula due to 
\cite{gine:goetze:mason:1997}:
\begin{equation}\label{eq:formulagine}
\E[ Y_{1}^{2m_1}\cdots Y_{r}^{2m_r} ] = \frac{1}{(k-1)!} \int_0^{\infty} 
\la^{k-1} (\E[\e^{-\la X^2}])^{n-r}  \prod_{j=1}^r \E[X^{2m_j}\e^{-\la X^2} ] \, \dint \la\,,
\end{equation}
where $1\le r\le k$, $m_1+\cdots+m_r=k$ and $m_i\ge 1$.
\par
We present some examples of distributions of $X$  which satisfy \eqref{eq:assumptionq}. 

\begin{example}[Standard normal distribution]\label{ex:normal}{\em
Assume $X_i \sim N(0,1)$. We calculate $\E[Y_{1}^{2m_1}\cdots Y_{r}^{2m_r} ]$ for the standard normal distribution via \eqref{eq:formulagine}. Since $X_1^2$ has $\chi^2$-distribution we know for $\lambda\ge 0$ that $\E[\e^{-\la X^2}]=(1+2\la)^{-1/2}$.
We have 
\begin{equation*}
\frac{\dint^m}{\dint \la^m} \e^{-\la X^2} = (-X^2)^m \e^{-\la X^2}\,.
\end{equation*}
Calculation yields 
\begin{equation}\label{eq:diffmgf}
\E[X^{2m}\e^{-\la X^2} ]
= (2m-1)!! (1+2\la)^{-1/2-m}\,.
\end{equation}
By \eqref{eq:formulagine} and \eqref{eq:diffmgf}, we have for $\ell\le k$
\begin{equation*}
\begin{split}
\E[Y_{1}^{2m_1}\cdots Y_{r}^{2m_r} ] &= \frac{1}{(\ell-1)!} \int_0^{\infty} 
\la^{\ell-1} (\E[\e^{-\la X^2}])^{n-r}  \prod_{j=1}^r \E[X^{2m_j}\e^{-\la X^2} ] \, \dint \la\\
&= \frac{1}{(\ell-1)!} \int_0^{\infty} 
\la^{\ell-1} (1+2\la)^{-(n+2\ell)/2}  \, \dint \la \prod_{j=1}^r (2m_j-1)!!\,.
\end{split}
\end{equation*}
Since 
\begin{equation*}
\int_0^{\infty} 
\la^{\ell-1} (1+2\la)^{-(n+2\ell)/2}  \, \dint \la = \frac{\Gamma(n/2) \Gamma(\ell)}{2^{\ell} \Gamma(n/2+\ell)}\,,
\end{equation*}
one obtains
\begin{equation}\label{eq:momentsnormal}
\E[ Y_{1}^{2m_1}\cdots Y_{r}^{2m_r}] = \frac{\Gamma(n/2) }{2^\ell \Gamma(n/2+\ell)} \prod_{j=1}^r (2m_j-1)!!\,,
\end{equation}
which allows one to conclude that
\begin{equation*}
\begin{split}
\frac{\E[ Y_{1}^{2m_1}\cdots Y_{r}^{2m_r}]}{\E[ Y_{1}^{2m_1}\cdots Y_{r}^{(2m_r-2)}]}&=
\frac{2m_r-1}{n+2\ell-2}\le \frac{2k}{n}\,,
\end{split}
\end{equation*}
where we used $m_r\le \ell\le k$. Hence \eqref{eq:assumptionq} holds with $q_n=2k_n$.
}\end{example}

\begin{example}[Gamma distribution]\label{ex:gamma}{\em
 Assume $X^2\sim \text{Gamma}(\alpha,\beta)$, $\alpha,\beta>0$. In this case
\begin{equation*}
\frac{\dint^{m}}{\dint \la^{m}} \E[\e^{-\la X^2}]=\frac{\dint^{m}}{\dint \la^{m}} \Big(1+\frac{\la}{\beta}\Big)^{-\alpha} = \frac{\Gamma(1-\alpha )}{\Gamma(1-\alpha -m)} \beta^{-m} \Big(1+\frac{\la}{\beta}\Big)^{-\alpha-m}\,.
\end{equation*}
For $\ell\le k$ one can calculate
\begin{equation*}
\begin{split}
\E[ Y_{1}^{2m_1}\cdots Y_{r}^{2m_r}] 
&= \frac{1}{(\ell-1)!} \int_0^{\infty} 
\la^{\ell-1} (\E[\e^{-\la X^2}])^{n-r}  \prod_{j=1}^r (-1)^{m_j} \frac{\dint^{m_j}}{\dint \la^{m_j}} \E[\e^{-\la X^2}]  \, \dint \la\\
&= \frac{\Gamma(\alpha n)(-1)^{\ell}}{\Gamma(\alpha n +\ell)} \prod_{j=1}^r  \frac{\Gamma(1-\alpha )}{\Gamma(1-\alpha -m_j)}\,.
\end{split}
\end{equation*}
Similarly to the previous example, \eqref{eq:assumptionq} holds with $q_n=(k_n+\alpha)/\alpha$.
}\end{example}

\section{Main results}\label{sec:3}\setcounter{equation}{0}
Our first result identifies the limit of the empirical spectral \ds\ 
$F_{\bfR}$ of the sample correlation matrix $\bfR$
for iid random fields $(X_{it})$ with generic element $X$.

\begin{theorem}[Limiting spectral distribution]\label{thm:mpcorrelation}
Assume the condition \eqref{Cpn}.
\begin{enumerate}
\item[(1)] If $X$ is centered and \eqref{eq:condX} holds
then $F_{\bfR}$ converges weakly to the \MP law given in \eqref{eq:MP}.
\item[(2)]
If $X$ is symmetric and  \eqref{eq:condX} does not hold, i.e.,
$\liminf_{\nto} n\,\E [ Y^4 ] >0$,
then 
\beao
\liminf_{\nto} \E\Big[\int x^k F_{\bfR}(dx)\Big]> \beta_k(\gamma)\,,\qquad k\ge 4\,,
\eeao
where $\beta_k(\gamma)$ is the $k$-th moment of the \MP law given in \eqref{eq:momentsmp}.
\end{enumerate}
\end{theorem}
The proof of parts (1) and (2) will be given in Sections~\ref{sec:5.1} and \ref{sec:5.5}, respectively.
\begin{remark}{\em 
Part (1) with condition \eqref{eq:condX} replaced by $\E[X^2]<\infty$ was proved in \cite{jiang:2004}. Later,
in \cite{bai:zhou:2008} the finite variance assumption was replaced by the weaker condition that
the distribution of $X$ belongs to the domain of attraction of the normal law. We discussed in the previous
section that   \eqref{eq:condX} holds under the latter condition.
Part (2) shows that \eqref{eq:condX} is the minimal condition for part (1) if $X$ is symmetric. By Lemma B.1 in \cite{bai:silverstein:2010},
$\lim_{\nto} \E\big[\int x^k F_{\bfR}(dx)\big]=\beta_k(\gamma)\,, k\ge 1\,,$
implies weak convergence of $F_{\bfR}$ to the \MP distribution as the latter is uniquely determined by its moments $(\beta_k(\gamma))_{k\ge 1}$.
}\end{remark}

If $X$ is symmetric, $n\,\E[Y^4]=o(1)$ and $p/n\to 0$, a slight modification of the proof of part (2) combined with the method of moments yields $F_{\bfR}\to \1_{[1,\infty)}$ weakly. Consequently, 
for any $\vep \in(0,1)$ the number of eigenvalues outside $(1-\vep,1+\vep)$ is $o(p)$ \as~ In particular, if $p$ is fixed, then $\mu_{(1)}$ and $\mu_{(p)}$ converge to $1$ \as~ In view of part (2), one concludes that $n\E[Y^4]=o(1)$ is a necessary and sufficient condition for the a.s. convergence of the eigenvalues $(\mu_{(i)})$ if $X$ is symmetric and $p$ fixed.  

When $p\to \infty$ one has to deal with the potentially $o(p)$ eigenvalues outside the support of the limiting spectral distribution. We develop a method to overcome this problem at the expense of strengthening the assumption $n\,\E[Y^4]=o(1)$ to \eqref{eq:assumptionq}.

A Borel--Cantelli argument to obtain an upper bound for $\limsup_n \mu_{(1)}$ requires an adequate bound on $\E[\mu_{(1)}^{k_n}]$, where $k_n\to \infty$. To this end, we use the inequality
\begin{equation*}
\E[\mu_{(1)}^{k_n}]\le \E[\tr \bfR^{k_n}] 
= \sum_{i_1,\ldots,i_{k_n}=1}^p \sum_{t_1,\ldots,t_{k_n}=1}^n \E[ Y_{i_1t_{k_n}} Y_{i_1t_1} Y_{i_2t_1}Y_{i_2t_2} \cdots Y_{i_{k_n}t_{{k_n}-1}} Y_{i_{k_n}t_{k_n}}  ]
\end{equation*}
and determine those summands on the \rhs~which are largest when weighted by their multiplicities. Using our {\em Path-Shortening Algorithm}, which is a novel technique that efficiently uses the inherent structure of sample correlation matrices, their contribution is calculated explicitly. The other summands can --with considerable technical effort-- be controlled by \eqref{eq:assumptionq}. 
Note that because of the identity  $\E[\tr \bfR^{k_n}] =p\, \E\big[\int x^{k_n} F_{\bfR}(dx)\big]$ the behavior of moments of the empirical spectral distribution is closely linked to the above upper bound.

The following result provides general conditions for the a.s.~\con\ of 
the largest and smallest eigenvalues $\mu_{(1)}$ and $\mu_{(p)}$ of $\bfR$ to the endpoints of the 
\MP law. The proof of this result is given in Section~\ref{sec:5.2}.
\begin{theorem}[Limit of extreme eigenvalues]\label{thm:mu1convergence}
Assume \eqref{Cpn}. 
\begin{itemize}
\item[(1)] If $\E[X^4]<\infty$ and $\E[X]=0$ 
\item[(2)] or $X$ is symmetric and satisfies condition \eqref{eq:assumptionq},
\end{itemize}
then 
\begin{equation}\label{eq:mlmgd}
{\mu}_{(1)} \to (1+\sqrt{\gamma})^2 \, \quad \as\,,
\end{equation}
\begin{equation}\label{eq:limitmup}
{\mu}_{(p)} \to (1-\sqrt{\gamma})^2 \, \quad \as\phantom{\,,}
\end{equation}
\end{theorem}

\begin{remark}{\em Part (1) was proved in \cite{jiang:2004,xiao:zhou:2010}; see the discussion
in Section~\ref{sec:sdlgj}. 
Theorem~\ref{thm:mu1convergence} indicates that the a.s. convergence of the extreme eigenvalues of $\bfR$
does not depend on the finiteness of the fourth or even second moment; see also the discussion in section \ref{sec:cond}. 
This is in stark contrast to the a.s. behavior of $n^{-1}\lambda_{(1)}$, the largest eigenvalue of the sample covariance matrix 
$n^{-1}\bfX\bfX'$. Note that there is a phase transition of the a.s. \asy\ behavior of the  
extreme eigenvalues at the border between finite and infinite fourth moment of $X$, 
while such a transition occurs for the empirical spectral distribution at the border between finite and infinite variance.}
\end{remark}
\subsection{Simulation study}\label{sec:simulation}
\begin{figure}[htb!]
  \centering
  \subfigure[Sample correlation: $\mu_{(p)}=0.086, \mu_{(1)}=2.898$.]{
    \includegraphics[trim = 1.5in 3.4in 1.5in 3.6in, clip, scale=0.5]{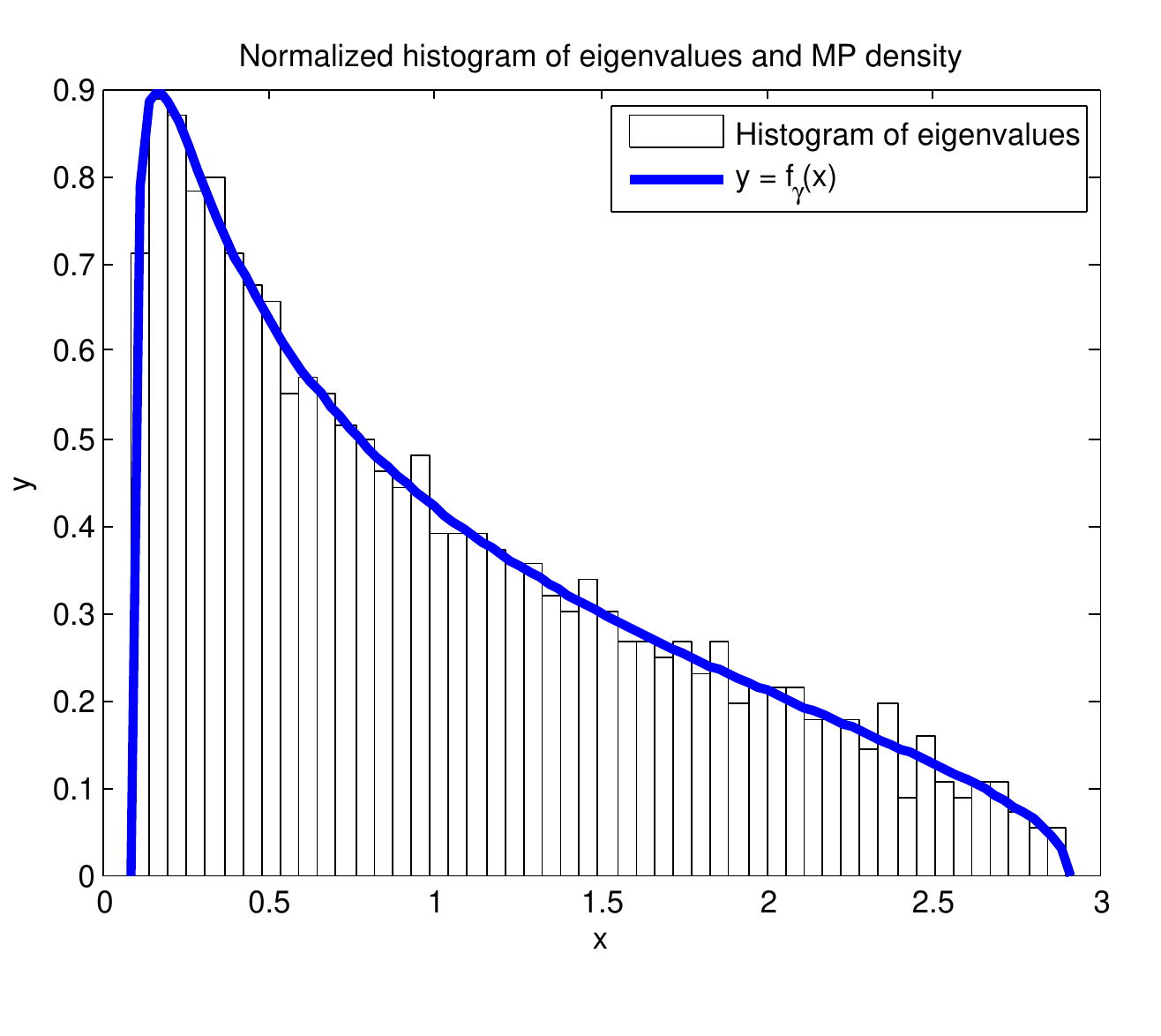}
  }
	\qquad
  \subfigure[Sample covariance: $\lambda_{(p)}/(n \E{[X^2]})= 0.085,\lambda_{(1)}/(n \E{[X^2]})=2.908$.]{
    \includegraphics[trim = 1.5in 3.4in 1.5in 3.6in, clip, scale=0.5]{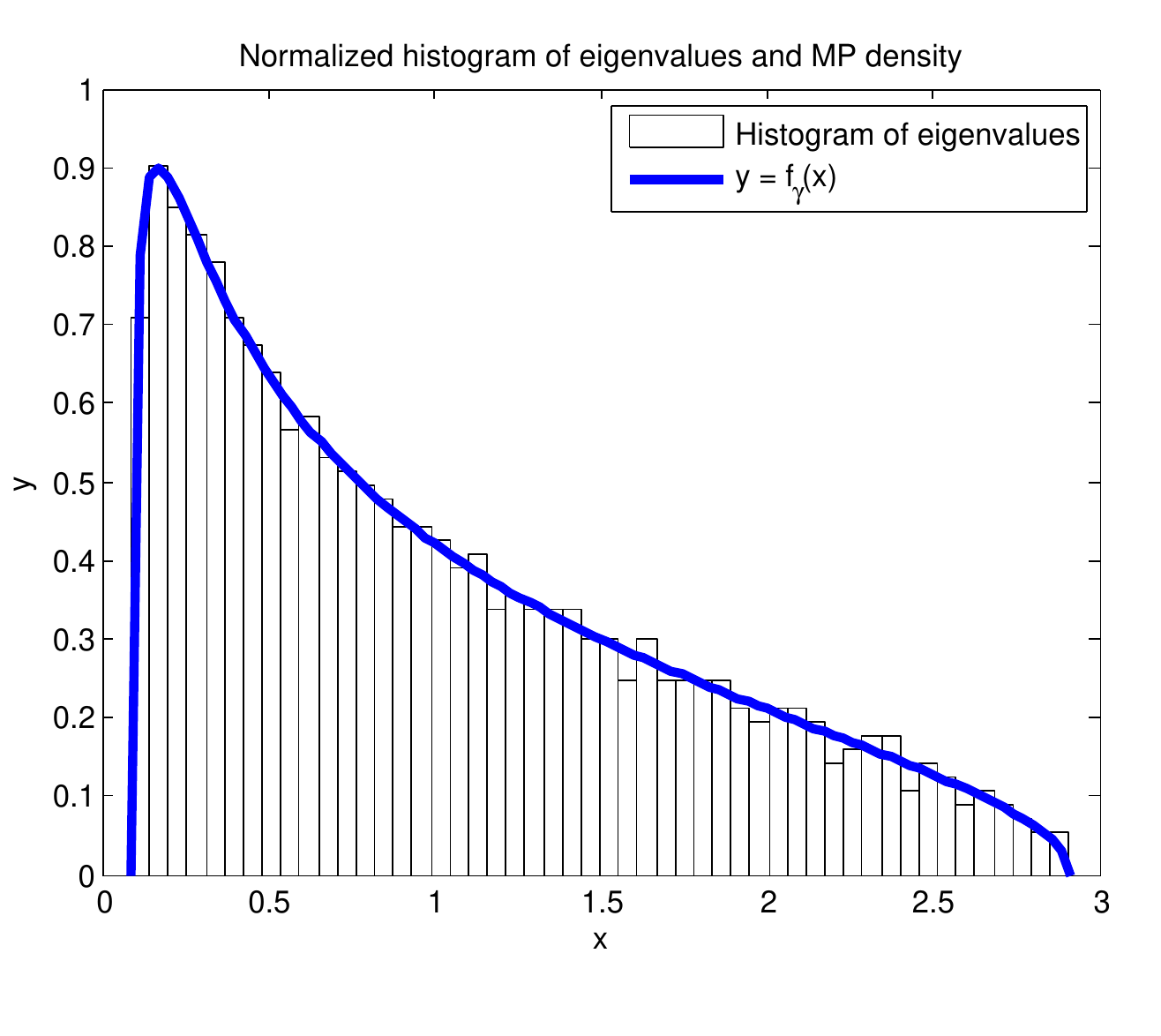}
  }
\caption{Histogram and \MP density for $X \sim t_6$, $n=2000, p=1000$.  $\gamma=0.5, (1-\sqrt{\gamma})^2=0.085, (1+\sqrt{\gamma})^2=2.914$.}
  \label{fig:alpha6}
\end{figure}

\begin{figure}[htb!]
  \centering
\subfigure[Sample correlation: $\mu_{(p)}=0.088, \mu_{(1)}=2.880$.]{
    \includegraphics[trim = 1.5in 3.4in 1.5in 3.6in, clip, scale=0.5]{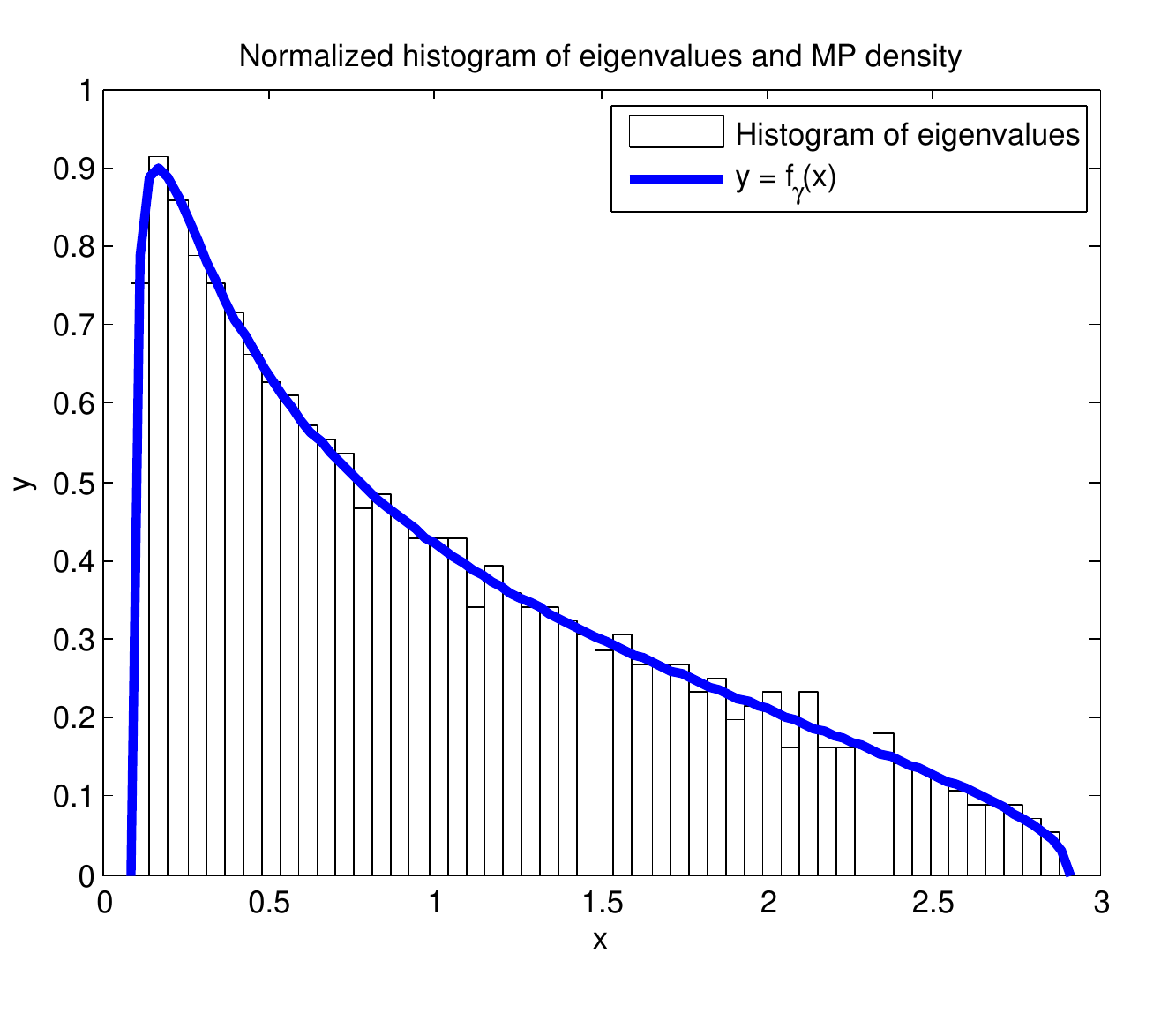}
		}
		\qquad
	\subfigure[Sample covariance: $\lambda_{(p)}/(n \E{[X^2]})= 0.083,\lambda_{(1)}/(n \E{[X^2]})=8.870$.]{
    \includegraphics[trim = 1.5in 3.4in 1.5in 3.6in, clip, scale=0.5]{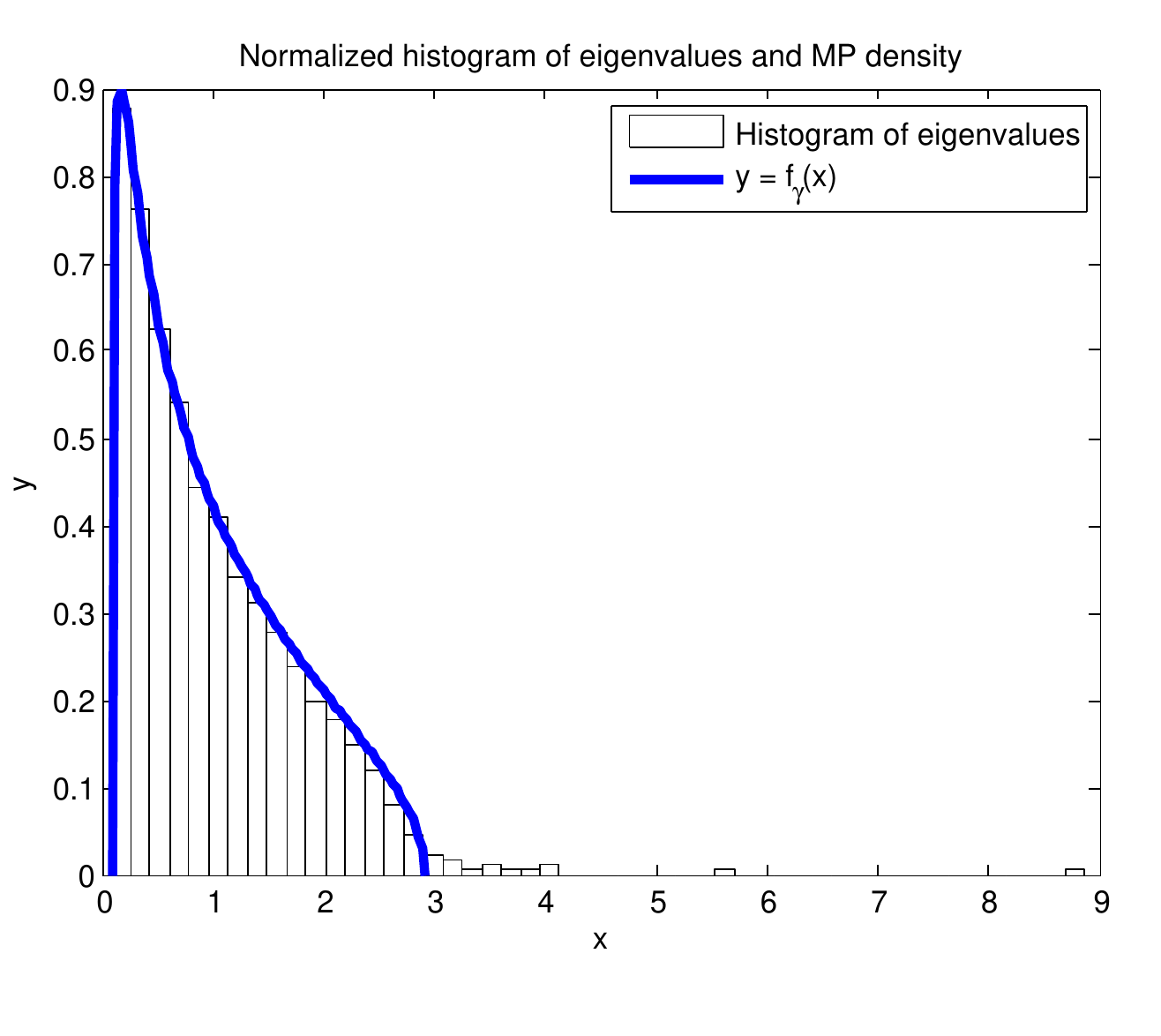}
  }
\caption{Histogram and \MP density for $X \sim t_3$, $n=2000, p=1000$. Here $\gamma=p/n=0.5, (1-\sqrt{\gamma})^2=0.085, (1+\sqrt{\gamma})^2=2.914$.}
  \label{fig:alpha3}
\end{figure}

\begin{figure}[htb!]
  \centering
\subfigure[Sample correlation: $\mu_{(p)}=0.086, \mu_{(1)}=2.902$.]{
    \includegraphics[trim = 1.5in 3.4in 1.5in 3.6in, clip, scale=0.5]{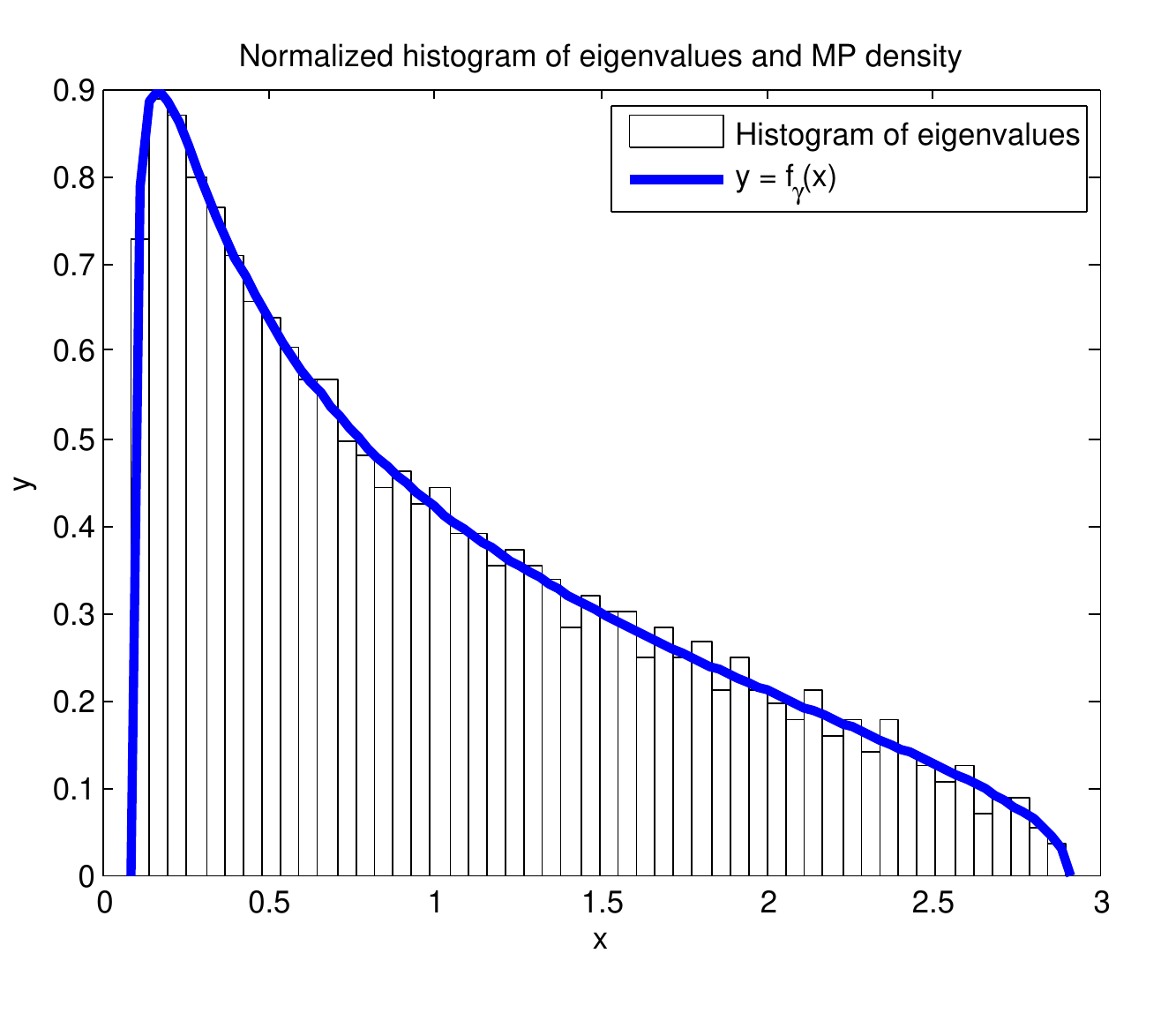}
  }
	\qquad
  \subfigure[Sample covariance: $\lambda_{(p)}/(n \E{[X^2]})=0.083, \lambda_{(1)}/(n \E{[X^2]})=3.176$.]{
    \includegraphics[trim = 1.5in 3.4in 1.5in 3.6in, clip, scale=0.5]{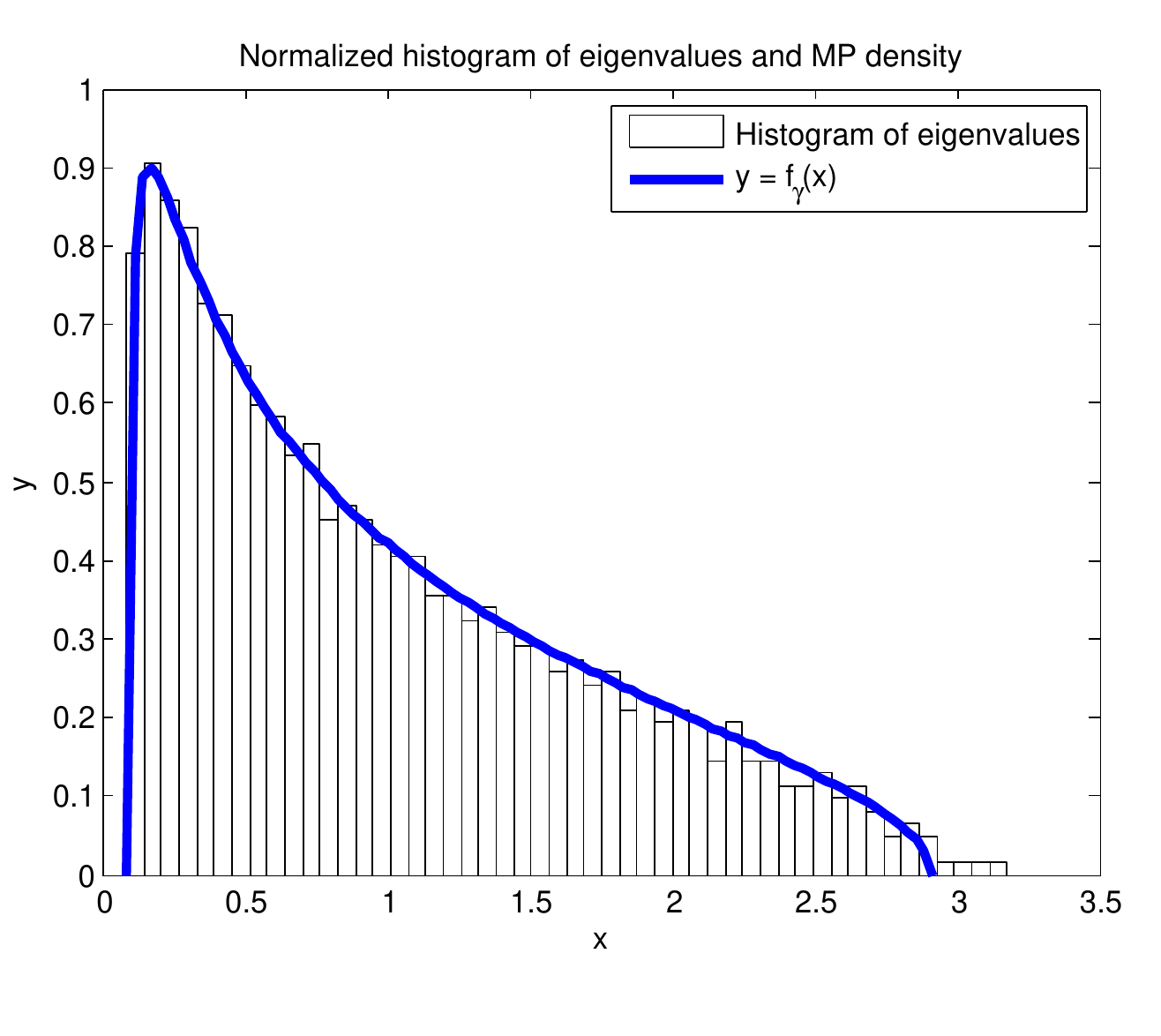}
  }
	\caption{Histogram and \MP density: $X \eid Z_1-Z_2$ for $Z_i\sim \text{Pareto}(3.99)$, $n=2000, p=1000$. Here $\gamma=p/n=0.5, (1-\sqrt{\gamma})^2=0.085, (1+\sqrt{\gamma})^2=2.914$.}
	\label{fig:pareto399}
	\subfigure[Sample correlation: $\mu_{(p)}=0.469, \mu_{(1)}=1.731$.]{
    \includegraphics[trim = 1.5in 3.4in 1.5in 3.6in, clip, scale=0.5]{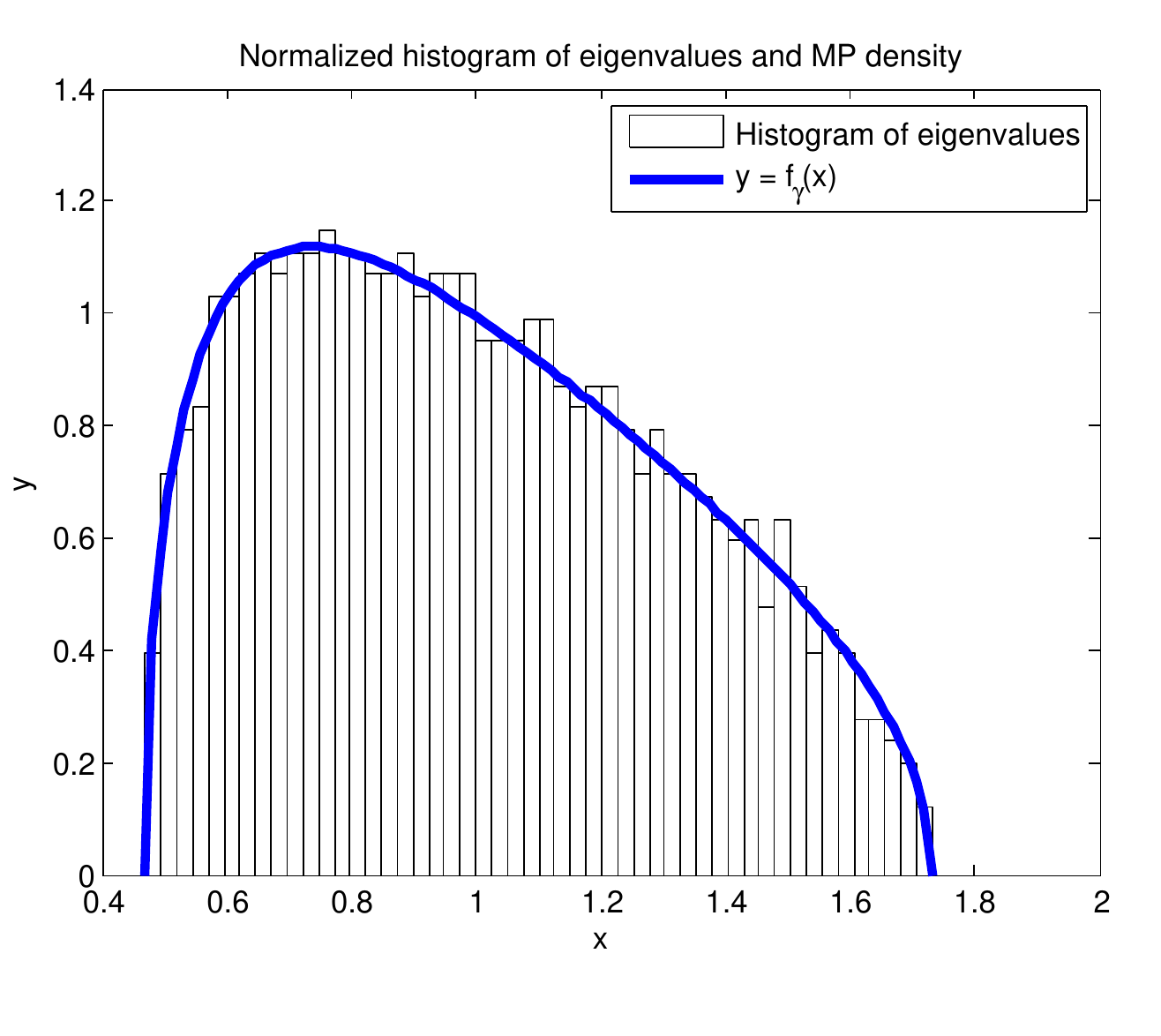}
  }
	\qquad
  \subfigure[Sample covariance: $\lambda_{(p)}/(n \E{[X^2]})=0.159, \lambda_{(1)}/(n \E{[X^2]})=35.319$.]{
    \includegraphics[trim = 1.5in 3.4in 1.5in 3.6in, clip, scale=0.5]{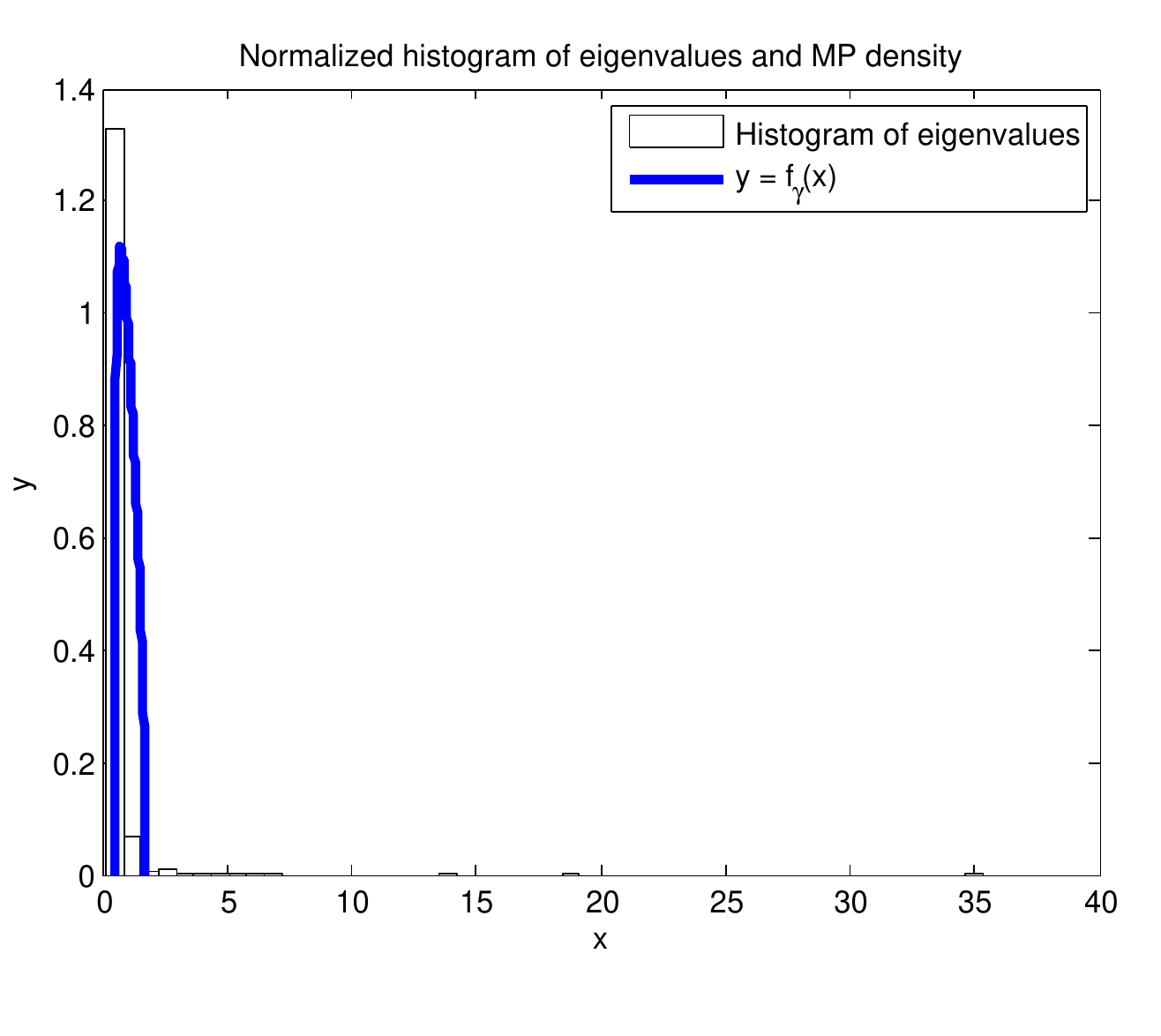}
  }
\caption{Histogram and \MP density for $X \sim t_{2.1}$, $n=10000, p=1000$. Here $\gamma=p/n=0.1, (1-\sqrt{\gamma})^2=0.467, (1+\sqrt{\gamma})^2=1.732$.}
  \label{fig:stability}
\end{figure}

\begin{figure}[htb!]
  \centering
\subfigure[$X \eid Z^2-\E Z^2$ for $Z\sim t_{1.5}$, $n=2000, p=1000, \gamma=0.5$.]{
    \includegraphics[trim = 1.5in 3.4in 1.5in 3.6in, clip, scale=0.5]{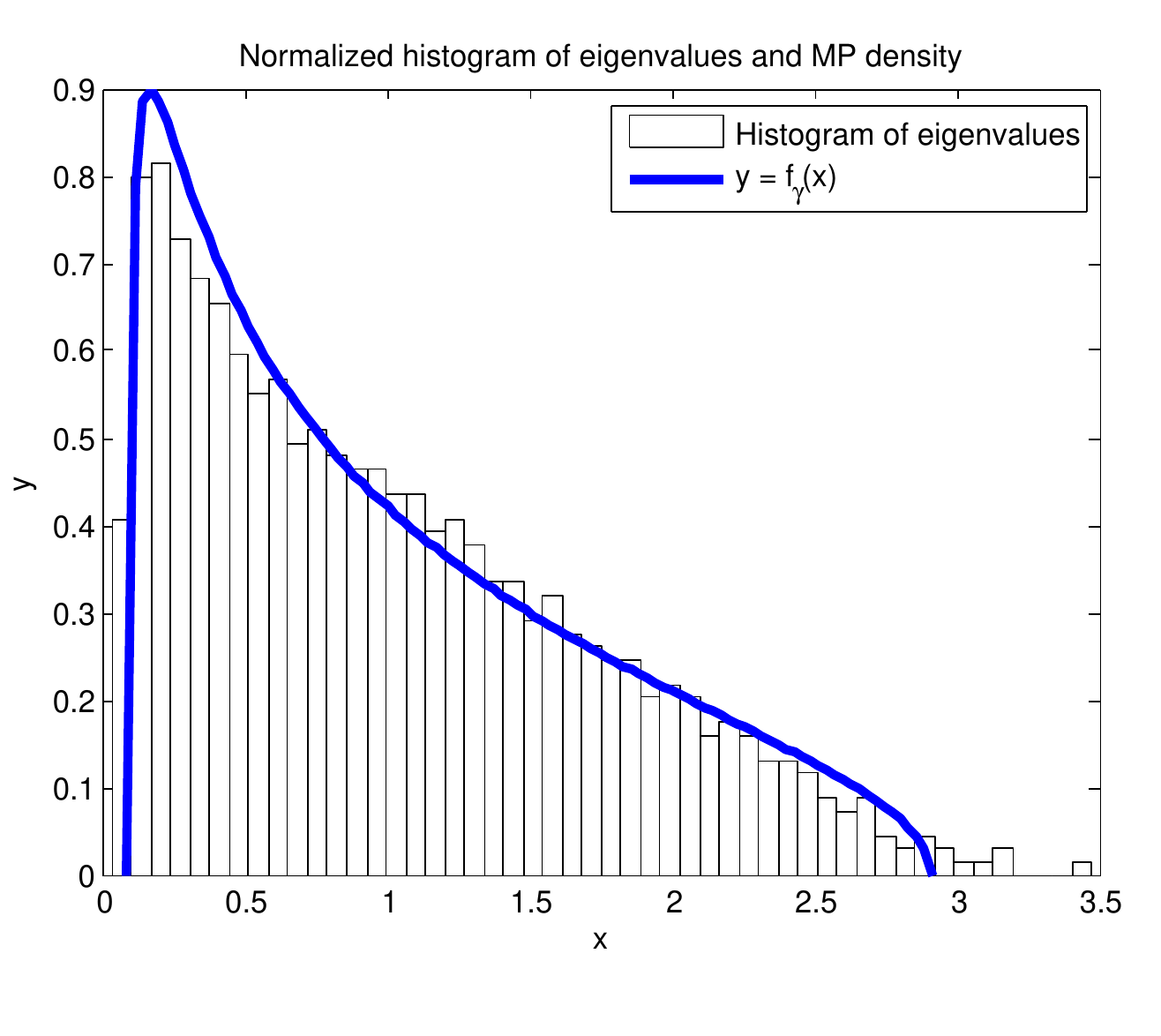}
  }
	\qquad
\subfigure[$X \sim t_{1.8}$, $n=10000, p=1000, \gamma=0.1$]{
    \includegraphics[trim = 1.5in 3.4in 1.5in 3.6in, clip, scale=0.5]{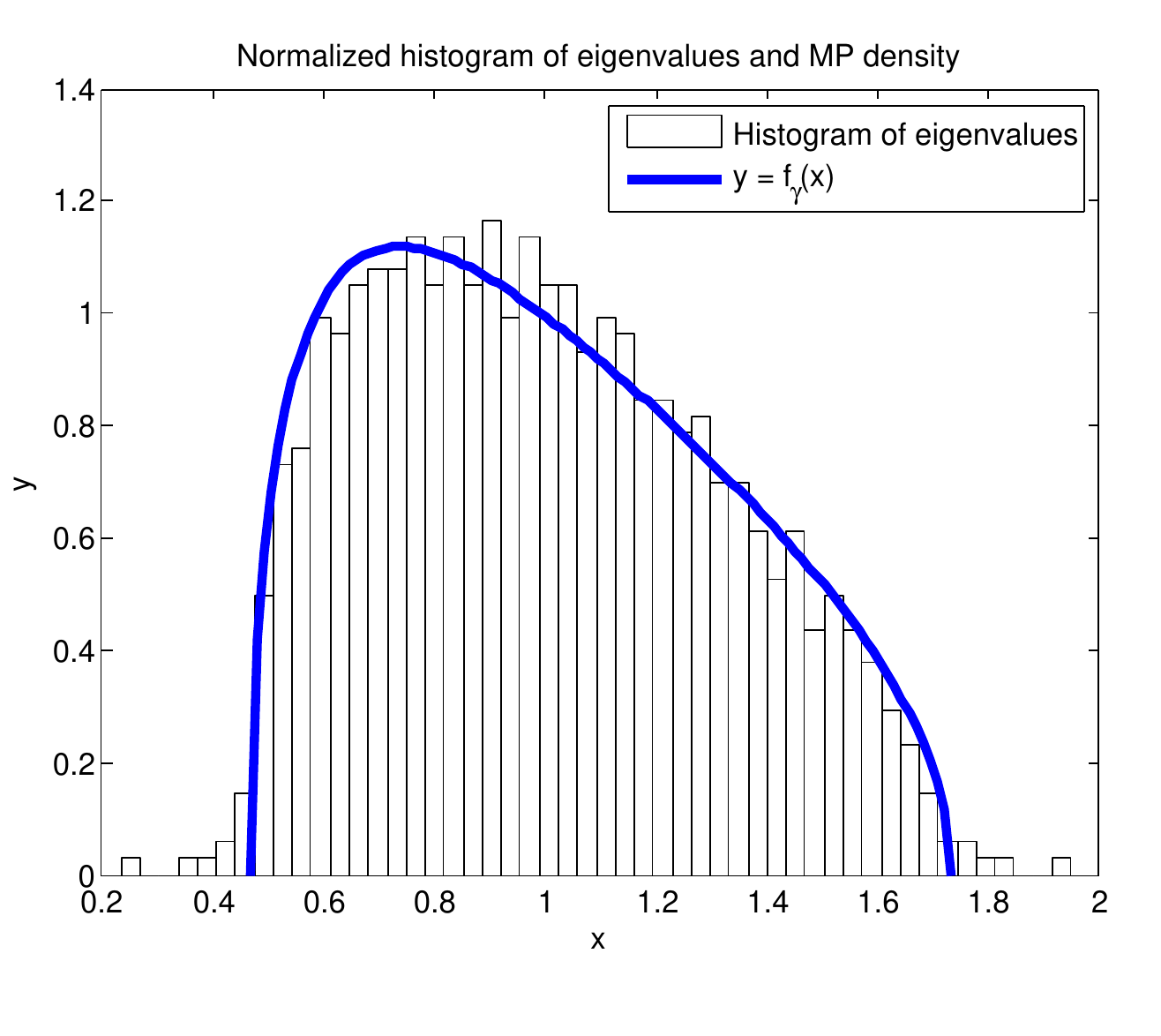}
  }
\caption{Histogram of $(\mu_{(i)})$ and \MP density}
  \label{fig:smallalpha}
\end{figure}

In this subsection we simulate a large data matrix $\X$ 
of iid entries. We compare the spectra of $\X\X'/n$ and $\bfR$ to the limiting \MP spectral density with appropriate 
parameter $\gamma$; see Theorem \ref{thm:mpcorrelation}. We simulate from different distributions of $X$ and choose various values for $p$ and $n$ to cover \MP distributions of several shapes. 
In what follows, we assume $\E[X^2]=1$, whenever the second moment is finite. 
\par

In Figure~\ref{fig:alpha6} we simulated a $1000\times 2000$ data matrix $\X$ with iid entries drawn from a $t_6$-distribution which we renormalized to meet the requirement $\E[X^2]=1$. To illustrate the weak convergence of $(F_{\bfR})$ and $(F_{\X\X'/n})$ we plot the histograms of the eigenvalues $(\mu_{(i)})$ and $(\lambda_{(i)}/n)$ and compare them to the \MP distribution with $\gamma=1/2$. 
As expected in the case $\E[X^4]<\infty$, the values $n^{-1} \la_{(1)}=2.9086$ and $n^{-1} \la_{(p)}=0.0855$ are very close to their theoretical almost sure limits $2.9142$ and $0.0858$, respectively. The same is valid for $\mu_{(1)}$ and $\mu_{(p)}$.
\par

In Figure~\ref{fig:alpha3} we simulate $X$ from a renormalized $t_3$-distribution with unit variance. The histograms of $(\mu_{(i)})$ and $(\lambda_{(i)}/n)$ resemble the corresponding \MP density $f_{1/2}$. Note that $\lambda_{(1)}/n$ can be larger than the right endpoint $(1+\sqrt{\gamma})^2$ since it has a different limit behavior than in the case $\E[X^4]<\infty$, while $\mu_{(p)}$ and $\mu_{(1)}$ are close to the endpoints $(1-\sqrt{\gamma})^2$ and $(1+\sqrt{\gamma})^2$, respectively, for which Theorem~\ref{thm:mu1convergence} provides a formal justification.  
\par

In Figures~\ref{fig:pareto399} and \ref{fig:stability} we simulated from distributions with infinite fourth moment. 
We drew from a symmetrized Pareto distribution with parameter $3.99$ to create the plots in Figure~\ref{fig:pareto399}. 
Note that in this case $\E[X^{3.99}]=\infty$, while $\E[X^{3.99-\vep}]<\infty$ for any $\vep >0$, i.e., 
we are at the ``border'' between finite and infinite fourth moment. 
The extreme eigenvalues in the sample correlation case are very close to their theoretical limits stated in Theorem~\ref{thm:mu1convergence}, whereas 
the largest eigenvalues of the sample covariance matrix cease to lie within the support of the \MP distribution. 
Note that the assumption $\E[X^2]=1$ is superfluous for the sample correlation plots due to self-normalization. For the histogram of $(\lambda_{(i)}/n)$ the knowledge of the correct value $\E[X^2]$ is crucial since, for instance, $\lambda_{(1)}/n \to (1+\sqrt{\gamma})^2 \E[X^2]\, \as$ 
In applications, $\E[X^2]$ needs to be estimated first and estimation errors might significantly alter the conclusion. One can easily imagine that Figure~\ref{fig:alpha6}(b) with a misspecified variance of the data could have resembled Figure~\ref{fig:pareto399}(b). In this respect sample correlations are more robust. 
\par

In Figure~\ref{fig:stability}, we choose $X$ from the standardized
$t_{2.1}$-\ds , moving closer to the infinite variance case. The histogram of $(\mu_{(i)})$ fits the \MP density very well and the extreme eigenvalues are located in a close proximity of the endpoints of the \MP support. The sample covariance case in (b) does not 
look particularly appealing due to the fact that there are a few relatively large eigenvalues. For example, $\lambda_{(1)}/n=35.3196$ while $(1+\sqrt{\gamma})^2$ is only $1.7325$. By \cite{auffinger:arous:peche:2009,heiny:mikosch:2015:iid,davis:mikosch:heiny:xie:2015}, the properly normalized $\lambda_{(1)}$ converges to a \Frechet distributed random variable. The correct normalization is roughly $n^{4/2.1}$ and hence it is expected that $\lambda_{(1)}/n$ is separated from the bulk, whose behavior ultimately determines the limiting spectral distribution, which is the \MP law with parameter $\gamma=0.1$. However, due to the separation between the top eigenvalues and the bulk, it is not obvious from a histogram with only $50$ classes that the \MP law provides a good fit to the spectral distribution in (b). This different behavior of sample correlations and covariances is an additional argument for the higher stability of results obtained from an analysis of the sample correlation matrix.  
\par

Finally, we turn to the case where the assumptions of Theorem~\ref{thm:mpcorrelation} are violated. We present two histograms of $(\mu_{(i)})$ with $\E[X^2]=\infty$ in Figure~\ref{fig:smallalpha}. 
In (a), we choose the non-symmetric $X \eid Z^2-\E Z^2$ for $Z\sim t_{1.5}$. 
In (b), the simulated $X$ is standardized $t_{1.8}$. 
The plots look surprisingly stable, given that the empirical spectral distribution 
does not weakly converge to the \MP law; see Theorem~\ref{thm:mpcorrelation}(2). 
The extreme eigenvalues $\mu_{(1)}$ and $\mu_{(p)}$ are much further away from $(1-\sqrt{p/n})^2$ and $(1+\sqrt{p/n})^2$, respectively, than in all the other sample correlation histograms we have seen so far.

\subsection{A version of condition \eqref{eq:assumptionq} for finite fourth moment and beyond}

In applications it is quite challenging to check condition \eqref{eq:assumptionq} for a given distribution.  In principle, \eqref{eq:assumptionq} can be verified by direct calculation using \eqref{eq:formulagine}. However, for some classical distributions, such as the $t$- and symmetrized Pareto distributions, \eqref{eq:formulagine} requires the calculation of a high-dimensional integral which might not always be as easy to compute as in Examples \ref{ex:normal} and \ref{ex:gamma}. 
Due to the complexity of the calculations, this paper does not contain a nicely worked out example of a distribution with infinite fourth moment that satisfies  \eqref{eq:assumptionq}.

In the existing literature the almost sure convergence of the extreme eigenvalues was established for finite fourth moment only. 
Section \ref{sec:simulation} indicates that the result holds for infinite fourth moment too. A necessary and sufficient condition, however,  is unknown. Our condition \eqref{eq:assumptionq} is a step towards the optimal condition which we believe to be $n \E[Y^4]\to 0$ for symmetric distributions. From our constructive method of proof one can see how $\E[Y^4]$ comes into play; see also Theorem \ref{thm:mpcorrelation} and especially the proof of part (2) for a better understanding how the behavior of $\E[Y^4]$ influences the limiting spectral distribution of $\bfR$.
\medskip

In this subsection we show that a version of condition \eqref{eq:assumptionq} holds for essentially all distributions with finite fourth moment. 
We also provide an explicit example of a distribution with infinite fourth moment such that $\mu_{(1)}\cip (1+\sqrt{\gamma})^2$ and $\mu_{(p)}\cip (1-\sqrt{\gamma})^2$. 

For a slowly varying sequence $\delta_n\to 0$ define $\widehat{X}_{it}^{(n)}=X_{it} \1_{\{|X_{it}|\le \delta_n \sqrt{n}\}}$ and set
\begin{equation}\label{eq:truncated}
\widehat{\X}_n= \big(\widehat{X}_{it}^{(n)}\big)_{i=1,\ldots,p;t=1,\ldots,n}\,.
\end{equation}
The sample correlation matrix of the truncated data $\widehat{\bfR}=\widehat{\Y}\widehat{\Y}'$ with $\widehat{Y_{it}}=\widehat{X}^{(n)}_{it}/\sqrt{\sum_{t=1}^n (\widehat{X}_{jt}^{(n)})^2}$ and 
\begin{equation}\label{eq:corrRhat}
\widehat{\bfR}= \big(\widehat{R}_{ij}\big) = \Big(\sum_{t=1}^n \frac{\widehat{X}_{it}^{(n)}\widehat{X}_{jt}^{(n)}}{\sqrt{\sum_{t=1}^n (\widehat{X}_{it}^{(n)})^2} \sqrt{\sum_{t=1}^n (\widehat{X}_{jt}^{(n)})^2}} \Big)\,, \quad i,j=1,\ldots,p\,,
\end{equation}
has largest and smallest eigenvalues $\widehat{\mu}_{(1)}$ and $\widehat{\mu}_{(p)}$, respectively. If $\E[X^{4}] <\infty$, then by Lemma 2.2 in \cite{bai:yin:krishnaiah:1988}, $\P(\X_n  \neq \widehat{\X}_n \text{ i.o.})=0$, which implies
\begin{equation*}
\max_{1\le i\le p} \{ |\mu_{(i)}-\widehat{\mu}_{(i)}|\} \cas 0\,.
\end{equation*}

Next, we introduce a condition similar to \eqref{eq:assumptionq}.
For ease of notation we will again write $(\widehat{Y}_1, \ldots, \widehat{Y}_n)= (\widehat{Y}_{11}, \ldots, \widehat{Y}_{1n})$ and $\widehat{D}=\widehat{D}_1=\sum_{t=1}^n (\widehat{X}_{1t}^{(n)})^2$.

\begin{itemize}
\item[\phantom{.}] {\bf{Condition}} $(\widehat{C}_q)$: There exists a slowly varying sequence $\delta_n\to 0$ such that 
\begin{equation}\label{eq:gsdglsrk}
\lim_{\nto} n^2\, \P(|X|>\delta_n \sqrt{n}) =0\,
\end{equation}
and there exists a sequence $q=q_n\to \infty$ such that for some integer sequence $k=k_n$ with $k/\log n \to \infty$ we have $(k^3 q)/n \to 0$ and the moment inequality
\begin{equation}\label{eq:assumptionqhat}
\E[ \widehat{Y}_{1}^{2m_1}\cdots \widehat{Y}_{r}^{2m_r} ] \le 
 \frac{q_n}{n} \, \E[ \widehat{Y}_{1}^{2m_1}\cdots \widehat{Y}_{r-1}^{2m_{(r-1)}}\widehat{Y}_{r}^{(2m_r-2)} ]\, 
\end{equation}
holds for $1\le r\le \ell-1$ and any positive integers $m_1,\ldots,m_r$ satisfying $m_1+\cdots +m_r=\ell$, where $\ell \le k$.
\end{itemize}

\begin{remark}{\em
Loosely speaking, condition \eqref{eq:assumptionq} fails if the probability of one $Y_i^2$ being much larger than the sum of the other $Y_j^2$ is ``not small enough''. In some cases it might be easier to check condition \eqref{eq:assumptionqhat} instead of \eqref{eq:assumptionq} since one is allowed to include an extra truncation step which alleviates the aforementioned issue at the cost of the additional restriction \eqref{eq:gsdglsrk}.
}\end{remark}

We have the following version of Theorem \eqref{thm:mu1convergence} part (2).
\begin{theorem}\label{thm:mu1convergence00}
Assume \eqref{Cpn}. 
\begin{itemize}
\item[(2')] If $X$ is symmetric and satisfies condition $(\widehat{C}_q)$, then
\end{itemize}
\begin{equation}\label{eq:dsgfses}
\widehat{\mu}_{(1)} \cas (1+\sqrt{\gamma})^2 \quad \text{ and } \quad \widehat{\mu}_{(p)} \cas (1-\sqrt{\gamma})^2\,.
\end{equation}
\end{theorem}
\begin{proof}
The truncated entries $\widehat{X}_{it}$ satisfy condition \eqref{eq:assumptionq} and are symmetric. Following the proof of Theorem~\ref{thm:mu1convergence} under condition \eqref{eq:assumptionq} we get \eqref{eq:dsgfses}.
\end{proof}

Equation \eqref{eq:gsdglsrk} in condition $(\widehat{C}_q)$ is designed for comparison of the eigenvalues of $\bfR$ and $\widehat{\bfR}$. Under \eqref{eq:gsdglsrk}, one has $\P(\X_n  \neq \widehat{\X}_n)\le c\, n^2\, \P(|X|>\delta_n \sqrt{n}) \to 0$, which implies
\begin{equation}\label{eq:gsejkbfsejkfn}
\max_{1\le i\le p} \{ |\mu_{(i)}-\widehat{\mu}_{(i)}| \} \cip 0\,.
\end{equation}

Next we present a rather crude way to check $(\widehat{C}_q)$ for distributions with finite $(4+\vep)$ moment. 
\begin{corollary}\label{cor:trunc}
Let $\vep>0$. If $\E[X]=0$, $|X|\ge c>0$ and $\E[|X|^{4+\vep}]<\infty$, then condition $(\widehat{C}_q)$ is valid and $\delta_n>0$ can be chosen as any slowly varying sequence tending to $0$.
\end{corollary}
\begin{proof}
If $\E[X]=0$ and $\E[|X|^{4+\vep}]<\infty$, equation \eqref{eq:gsdglsrk} and the fact that $(\delta_n)$ can be chosen as any slowly varying sequence follow from the Borel--Cantelli lemma and the proof of Lemma 2.2 in \cite{bai:yin:krishnaiah:1988}.
Since $|X|\ge c>0$ and $\widehat{X}^2 \le \delta_n^2 n$ we get
\begin{equation}\label{eq:cr}
\begin{split}
\E[ \widehat{Y}_{1}^{2m_1}\cdots \widehat{Y}_{r}^{2m_r} ] &\le 
\delta_n^2 n \, \E[ \widehat{Y}_{1}^{2m_1}\cdots \widehat{Y}_{r-1}^{2m_{r-1}}\widehat{Y}_{r}^{2m_r-2} D^{-1}]\\
&\le \delta_n^2  \, \E[ \widehat{Y}_{1}^{2m_1}\cdots \widehat{Y}_{r-1}^{2m_{r-1}}\widehat{Y}_{r}^{2m_r-2}]/c\\
&= \frac{q_n}{n}  \, \E[ \widehat{Y}_{1}^{2m_1}\cdots \widehat{Y}_{r-1}^{2m_{r-1}}\widehat{Y}_{r}^{2m_r-2}]\,,
\end{split}
\end{equation}
for $q_n=n \delta_n^2/c$.
Choose $\delta_n=(\log n)^{-2}$ and $k_n=(\log n)^{1+\vep/3}$ with $\vep \in (0,1)$. We calculate
\begin{equation*}
\frac{k_n^3 q_n}{n}=c^{-1}\,(\log n)^{3+\vep-4} \to 0\,,
\end{equation*}
which finishes the proof.
\end{proof}

If $\E[X^4]=\infty$, a slowly varying sequence $(\delta_n)$ that satisfies \eqref{eq:gsdglsrk} does not always exist. For instance if $|X|$ is regularly varying with index $\alpha>0$, we have
\begin{equation*}
n^2 \P(|X| > \delta_n \sqrt{n})= n^{2-\alpha/2} \delta_n^{-\alpha} \ell(\delta_n \sqrt{n})
\end{equation*}
for some slowly varying function $\ell$.  For a slowly varying sequence $\delta_n\to 0$ we therefore have 
\begin{eqnarray}\label{eq:A}
\lim_{\nto} n^2 \, \P(|X| > \delta_n \sqrt{n}) =
\left\{\begin{array}{ll}
0 \,, & \mbox{if } \alpha>4\,, \\
\infty \,, & \mbox{if } \alpha<4\,.
\end{array}\right. 
\end{eqnarray}
As a consequence we need $\alpha=4$ and a particular relationship between $\ell$ and $\delta_n$ for both $\E[X^4]=\infty$ and \eqref{eq:gsdglsrk} to hold. An example of such a distribution is presented next.

\begin{example}[Weak limit of extreme eigenvalues under infinite fourth moment]\label{prop:exampleinfinite} {\em
Let $u>4$ and $s\in (4/u,1)$. Consider a symmetric random variable $X$ satisfying $|X|\ge c>0$ and
\begin{equation*}
\P(|X|>x)=x^{-4} (\log x)^{-s}\,, \quad x\ge x_0\,.
\end{equation*}
Then $\E[X^4]=\infty$. 
If $p/n \to \gamma \in (0,1]$, we have for the largest eigenvalue $\mu_{(1)}$ and the smallest eigenvalue $\mu_{(p)}$ of the sample correlation matrix $\bfR$:
\begin{equation}\label{eq:convprob}
\mu_{(1)} \cip (1+\sqrt{\gamma})^2 \quad \text{ and } \quad \mu_{(p)} \cip (1-\sqrt{\gamma})^2\,.
\end{equation}
Using the crude bounds in \eqref{eq:cr}, the almost sure convergence of the extreme eigenvalues is replaced by weak convergence.
}\end{example}
\begin{proof}[Proof of \eqref{eq:convprob}]
For the slowly varying sequence $\delta_n=(\log n)^{-u}$ and $\widehat{X}_{it}^{(n)}=X_{it} \1_{\{|X_{it}|\le \delta_n \sqrt{n}\}}$, we have due to $4u-s<0$,
\begin{equation*}
\begin{split}
\P(\X_n  \neq \widehat{\X}_n)&\le c n^2 \, \P(|X|>\delta_n \sqrt{n})\\
&\sim c\,2^{s} (\log n)^{4u-s} \to 0\,.
\end{split}
\end{equation*}
By the proof of Corollary~\ref{cor:trunc}, condition $(\widehat{C}_q)$ holds. 
By Theorem \ref{thm:mu1convergence00}, we have \eqref{eq:dsgfses}.
Together with \eqref{eq:gsejkbfsejkfn} this finishes the proof.
\end{proof}

\subsection{A remark on the centered sample correlation matrix}

We presented results for the matrices $\bfR$ and $\X\X'$, assuming that $\E[X]=0$ when $\E[|X|]<\infty$. 
In practice, the expectation of $X$ typically has to be estimated. We discuss what has to be changed in the
aforementioned theory in this case. We consider the matrix $\widetilde{\X}\widetilde{\X}'$, where 
\begin{equation*}
\widetilde{X}_{it}= X_{it}- \overline{X}_i \quad \text{ and } \quad \overline{X}_i= \frac{1}{n} \sum_{t=1}^n X_{it}\,.
\end{equation*}
and the corresponding correlation matrix $\widetilde{\bfR} = \widetilde{\bfF}^{1/2} \widetilde{\X}\widetilde{\X}'\widetilde{\bfF}^{1/2}$, where $\widetilde{\bfF}$ is the $p\times p$ diagonal matrix with entries 
\begin{equation*}
\widetilde{F}_{ii}= \frac{1}{(\widetilde{\X}\widetilde{\X}')_{ii}}\,, \quad i=1,\ldots,p.
\end{equation*}
\par
In contrast to \eqref{eq:lamu} an application of Weyl's inequality \cite{bai:silverstein:2010} yields
\begin{equation}\label{eq:sgsgrs}
n^{-1} |\la_{(1)}(\X\X')- \la_{(1)}(\widetilde{\X}\widetilde{\X}')| \le n^{-1} \twonorm{\X\X'-\widetilde{\X}\widetilde{\X}'}\,,
\end{equation}
where, in general, the \rhs\ does not converge to zero. 
However, since $\X-\widetilde{\X}$ is a rank $1$ matrix, it is known from \cite{bai:silverstein:2010} that $n^{-1} \X\X'$ and $n^{-1}  \widetilde{\X}\widetilde{\X}'$ share the same limiting spectral distribution (if it exists) with right endpoint $b$ say. Therefore we have
\begin{equation*}
\liminf_{\nto} \frac{\lambda_{(1)}(\widetilde{\X}\widetilde{\X}')}{n} \ge b\,.
\end{equation*}
Following \cite{elkaroui:2009}, we let $\mathbf{H}= \bfI_n-n^{-1} \mathbf{1}\mathbf{1}'$, where $\mathbf{1}=(1,\ldots,1)'$. Then one can write $\widetilde{\X}=\X \mathbf{H}$ and since $\mathbf{H}$ is a symmetric matrix with $(n-1)$ eigenvalues equal to $1$ and one eigenvalue equal to $0$ we see that
\begin{equation*}
\la_{(1)}(\widetilde{\X}\widetilde{\X}') \le \la_{(1)}(\X\X')\,.
\end{equation*} 
We conclude
\begin{equation*}
\lim_{\nto} \frac{\lambda_{(1)}(\widetilde{\X}\widetilde{\X}')}{n} = b\quad \as
\end{equation*}
whenever $\la_{(1)}(\X\X')/n \to b$ a.s. Therefore the a.s. behavior of the largest eigenvalues of $\bfX\bfX'$ and $\wt \bfX\wt\bfX'$ is
closely related. 

Due to the shift and scale invariance of sample correlations, the aforementioned arguments remain valid for the ordered eigenvalues 
\begin{equation*}
\widetilde{\mu}_{(1)}\ge \cdots \ge \widetilde{\mu}_{(p)}
\end{equation*}
of $\widetilde{\bfR}$ if $\E[X^4]<\infty$ and $\E[X]=c$ (not necessarily zero), as shown in \cite{jiang:2004}. Then we have $\widetilde{\mu}_{(1)} \to (1+\sqrt{\gamma})^2$~ \as ~ and $\widetilde{\mu}_{(p)} \to (1-\sqrt{\gamma})^2$ ~ \as~
In Theorem 2 of \cite{jiang:2004} it was proven that if $\E[X^2]< \infty$ and $p/n\to \gamma\in (0,\infty)$, the empirical spectral distribution of $\widetilde{\bfR}$ converges weakly to the \MP law.


\section{Technical results}\label{sec:5.3}\setcounter{equation}{0}
In this section we provide most technical results required for the proofs of the main theorems. 
We develop a new approach which efficiently uses the structure of sample correlation matrices.
The goal of this section is to prove Proposition~\ref{prop:paths}.
\par
Throughout $(X_{it})$ are iid symmetric, which implies that the $Y_{it}$ are symmetric as well. We will study the moments
\begin{equation*}
\sum_{t_1,\ldots,t_k=1}^n \E[ Y_{i_1t_k} Y_{i_1t_1}  Y_{i_2t_1} Y_{i_2t_2} Y_{i_3t_2} Y_{i_3t_3} \cdots Y_{i_kt_{k-1}} Y_{i_kt_k}  ].
\end{equation*}
for $k\ge 1$  and various choices of {\em paths} $I=(i_1, i_2, \ldots, i_k)\in \{ 1,\ldots, p\}^k$. In this case, $\length(I)=k$ is the {\em length of the path}.
We say that a {\em path} $(i_1, i_2, \ldots, i_k)$ is an $r$-{\em path} if it contains exactly $r$ distinct components.
A path is {\em canonical} if $i_1=1$ and $i_l \le \max \{i_1, \ldots, i_{l-1}\} +1, l\ge 2$. 
A canonical $r$-path satisfies $\{i_1, i_2, \ldots, i_k\} = \{1, \ldots, r\}$.
Two paths are {\em isomorphic} if one becomes the other by a suitable permutation on $(1,\ldots,p)$. Each {\em isomorphism class} contains exactly one canonical path.
For more details and examples of these path notions consult Section 3.1.2 in \cite{bai:silverstein:2010} and the references therein.
For $k\ge 1$, define
\beao
f(I,T)=\E[ Y_{i_1t_k} Y_{i_1t_1}  Y_{i_2t_1} Y_{i_2t_2} Y_{i_3t_2} Y_{i_3t_3} \cdots Y_{i_kt_{k-1}} Y_{i_kt_k}  ]\,,\quad I,T \in \{1,\ldots,k \}^k\,.
\eeao
Finally, define $F(\emptyset)=n$ and
\beao
F(i_1,\ldots,i_k)=F_n(i_1,\ldots,i_k)=\sum_{t_1,\ldots,t_k=1}^n f((i_1,\ldots,i_k),(t_1,\ldots,t_k))\,.
\eeao 
Note that $F(I_1)=F(I_2)$ if $I_1,I_2$ lie in the same isomorphism class. Therefore, whenever we are interested in 
$F(I)$ we can assume without loss of generality that $I$ is canonical. 
\par
In what follows, 
we will consider transformations of the path $I$ leading to a new path $S(I)$. 
For ease of notation, we will also assume $S(I)$ canonical. 
If it is not canonical, we can always work with its {\em canonical representative}, 
the unique canonical path in its isomorphism class. 
\par
 We will show that $F(I)$ can be often expressed as $F(S(I))$ times a certain power of $n$, where $S(I)$ is a shorter path than $I$. We start with two examples.
\begin{example}\label{ex:11}{\em
Let $I=(1,1,2,2)$ and recall that $(Y_{it})$ are symmetric. We have
\begin{equation*}
\begin{split}
F(1,1,2,2)&= \sum_{t_1,\ldots,t_4=1}^n \E[Y_{1t_4} Y_{1t_1}^2 Y_{1t_3}] \,\,\E[Y_{2t_2} Y_{2t_3}^2 Y_{2t_4}]\\
&= \sum_{t_2,t_4=1}^n \E\Big[Y_{1t_4} \underbrace{\sum_{t_1=1}^n Y_{1t_1}^2}_{=1} Y_{1t_3}\Big] \,\, \E\Big[Y_{2t_2} \underbrace{\sum_{t_3=1}^nY_{2t_3}^2}_{=1} Y_{2t_4}\Big]\\
&= \sum_{t_2=1}^n \underbrace{\E[Y_{1t_2}^2]}_{=n^{-1}}\underbrace{\E[Y_{2t_2}^2]}_{=n^{-1}}=n^{-1}\,.
\end{split}
\end{equation*}
Since $F(\emptyset)=n$ we get $F(1,1,2,2)=F(\emptyset) n^{-2}$, where $S(I)=\emptyset$ is interpreted as a path of length zero.
}\end{example}
\begin{example}\label{ex:12}{\em
For $I=(1,2,1,2,3,3)$ we get
\begin{equation*}
\begin{split}
F(1,2,1,2,3,3)&= \sum_{t_1,t_2,t_3,t_4,t_6=1}^n \E[Y_{1t_6} Y_{1t_1}Y_{1t_2} Y_{1t_3}] \,\,\E[Y_{2t_1}Y_{2t_2} Y_{2t_3} Y_{2t_4}] \,\,
\underbrace{\E\Big[Y_{3t_4} \sum_{t_5=1}^n Y_{3t_5}^2 Y_{3t_6}\Big]}_{=n^{-1} \1_{\{t_4=t_6\}}} \\
&= n^{-1} \sum_{t_1,\ldots,t_4=1}^n \E[Y_{1t_4} Y_{1t_1}Y_{1t_2} Y_{1t_3}] \,\,\E[Y_{2t_1}Y_{2t_2} Y_{2t_3} Y_{2t_4}]\\
&= n^{-1} F(1,2,1,2)\,.
\end{split}
\end{equation*}
Therefore we have $F(1,2,1,2,3,3)= F(S(I)) n^{-1}$ for the shorter path $S(I)=(1,2,1,2)$. 
}\end{example}

The ideas of Examples \ref{ex:11} and \ref{ex:12} will be formalized in the path-shortening algorithm and Lemma \ref{prop:PSI}.
When calculating values of $F$, the path-shortening function $PS$ will be useful. 
Let $I=(i_1,\ldots,i_k)\in \{1,\ldots,k \}^k$. $PS(I)$ is the output of the following algorithm.

\subsection*{Path-Shortening Algorithm $PS(I)$.}
\begin{itemize}
\item[Input: ] Path $I=(i_1,\ldots,i_k)$. Set $J=I$ and $R=0, \runs =0$. 
\item[Step 0:] Set $l= \length(I)$. Go to Step 1.
\item[Step 1:] Erase runs. 
\begin{itemize}
\item If $i_j=i_{j+1}$ for some $1\le j\le l$, where we interpret $i_{l+1}$ as $i_1$, erase element $i_j$ from the path. Set $I=(i_1, \ldots, i_{j-1}, i_{j+1}, \ldots, i_l)$, $\runs = \runs +1$ and return to Step 0.
\item Otherwise proceed with Step 2.
\end{itemize}
\item[Step 2:] Let $R_1$ be the number of elements of the path $I$ which appear exactly once. Set $R:=R+R_1$. Then define $I$ to be the resulting (possibly shorter) path which is obtained by deleting those $R_1$ elements from the path $I$. Go to Step 3.
\item[Step 3:] 
\begin{itemize}
\item If $J=I$, then return $(I,R,\runs)$ as output. 
\item If $J\neq I$, set $J:=I$ and return to Step 0.
\end{itemize}
\end{itemize} 

\begin{definition}\label{def:pathshortening}
The path-shortening function $PS$ is the output  $(S(I),R(I),\runs(I))$ of the Path-Shortening Algorithm (PSA) 
where $S(I)$ is the resulting shortened path and $R(I)$ is the total number of 
elements that were removed in Step 2 of the PSA. We write 
$PS(I)= (S(I),R(I),\runs(I))$.
\end{definition}

\subsection*{Properties of $PS(I)$.}
Clearly, ${\rm length}(S(I))\le {\rm length}(I)$. 
If $I=(1,\ldots,r)$ then $S(I)=\emptyset$, which shows that $S(I)$ can have length zero. 
Furthermore, all elements in $S(I)$ appear at least twice. If $I$ is an $r$-path then $R(I)\le r$.
\begin{lemma}\label{prop:PSI}
For any $I \in \{1,\ldots,k \}^k$, we have $F(I)=F(S(I))\, n^{-R(I)}$.
\end{lemma}
\begin{proof}
We shall look at the changes made to $I$ in Steps 1 and 2 of the PSA  separately.\\
{\em Assume we are in Step 1.}
\begin{itemize}
\item If $i_j=i_{j+1}$ for some $1\le j\le l$, where we interpret $i_{l+1}$ as $i_1$, erase element $i_j$ from the path. Set $S_1(I)=(i_1, \ldots, i_{j-1}, i_{j+1}, \ldots, i_l)$.
\item Otherwise, $S_1(I)=I$.
\end{itemize}
Since Step 1 does not influence the value of $R$ it suffices to show $F(I)=F(S_1(I))$. 
If $S_1(I)=I$ there is nothing to show. Therefore assume $i_j=i_{j+1}$ for some $j$. In this case,  we have
\beam\label{eq:proofstep1}
F(I)
&=& \sum_{t_1,\ldots,t_{j-1}, t_{j+1}, \ldots, t_k=1}^n \E[ Y_{i_1t_k} Y_{i_1t_1}\cdots Y_{i_{j-1}t_{j-2}} Y_{i_{j-1}t_{j-1}} 
Y_{i_{j}t_{j-1}} \sum_{t_j=1}^n Y_{i_{j}t_{j}}^2 Y_{i_{j}t_{j+1}}\nonumber\\
&& \qquad \qquad Y_{i_{j+2}t_{j+1}} Y_{i_{j+2}t_{j+2}} \cdots Y_{i_kt_{k-1}} Y_{i_kt_k}  ]\nonumber\\
&=& \sum_{t_1,\ldots,t_{j-1}, t_{j+1}, \ldots, t_k=1}^n \E[ Y_{i_1t_k} Y_{i_1t_1} \cdots Y_{i_{j-1}t_{j-2}} Y_{i_{j-1}t_{j-1}} 
Y_{i_{j}t_{j-1}}  Y_{i_{j}t_{j+1}}\nonumber\\
&& \qquad \qquad Y_{i_{j+2}t_{j+1}} Y_{i_{j+2}t_{j+2}} \cdots Y_{i_kt_{k-1}} Y_{i_kt_k} ]= F(S_1(I))\,,
\eeam
where we used $\sum_{t=1}^nY_{it}^2=1$. This proves that Step 1 poses no problem.\\[1mm]
{\em Next we turn to Step 2.} Without loss of generality we can assume that $I$ 
does not contain any runs. If all elements of $I$ appear at least twice there is nothing to prove. 
Therefore assume the $j$th element $i_j$ appears only once and $R_1=1$. 
Let $S_2(I)$ denote the path $I$ with the $j$th element removed. 
Thus we have to show $F(I)=F(S_2(I))n^{-1}$.  In this case,  we have
\beam\label{eq:proofstep2}
F(I)&=&\sum_{t_1,\ldots,t_{j-1}, t_{j+1}, \ldots, t_k=1}^n \sum_{t_j=1}^n \E[ Y_{i_{j}t_{j-1}} Y_{i_{j}t_{j}} ]\nonumber\\&&\qquad\qquad\times \E[ Y_{i_1t_k} Y_{i_1t_1} \cdots Y_{i_{j-1}t_{j-2}} Y_{i_{j-1}t_{j-1}}
Y_{i_{j+1}t_{j}} Y_{i_{j+1}t_{j+1}} \cdots Y_{i_kt_{k-1}} Y_{i_kt_k}  ]\nonumber\\
&=&\sum_{t_1,\ldots,t_{j-1}, t_{j+1}, \ldots, t_k=1}^n n^{-1} \E[ Y_{i_1t_k} Y_{i_1t_1} \cdots Y_{i_{j-1}t_{j-2}} Y_{i_{j-1}t_{j-1}}
Y_{i_{j+1}t_{j-1}} Y_{i_{j+1}t_{j+1}} \cdots Y_{i_kt_{k-1}} Y_{i_kt_k}  ]\nonumber\\
&=& F(S_2(I))n^{-1}\,,
\eeam
where first and third equality come from writing out the definition of $F$.
For the second equality we used that $t_{j-1}=t_j$ is necessary for $\E[ Y_{i_{j}t_{j-1}} Y_{i_{j}t_{j}} ]$ to be non-zero and $\E[ Y_{i_{j}t_{j}}^2 ]=n^{-1}$. If $R_1>1$ we can apply the above argument iteratively to obtain $F(I)=F(S_2 \circ \cdots \circ S_2(I))n^{-R_1}$. The proof is complete.
\end{proof}
Define for $k\ge 1$, a function $g$ by $g(\emptyset)=1$ and
\begin{equation}\label{eq:defmaxg}
g(I)= \max_{T \in \{1,\ldots,k\}^k}\{ |T|\,:\, f(I,T)>0  \}\,,\qquad I \in \{1,\ldots,k \}^k\,.
\end{equation}
From now on, we assume the $I$-paths to be canonical.
\begin{lemma}\label{lem:9.5}
Let $I$ be a canonical $r$-path of length $k$.  
For any $T \in \{1,\ldots,k \}^k$ such that $f(I,T)>0$ we have $|T|\le k-r+1$. 
\end{lemma}
\begin{proof}
Without loss of generality we may assume that $T$ is canonical. 
We shall sometimes refer to the $t_i$'s as $t$-indices. 
In the beginning one should think of the $t$-indices as pairwise distinct whenever possible. 
Their actual values are not relevant for the value $f(I,T)$. In all cases, except $r=1$, there are certain $t$-indices that have to coincide such that $f(I,T)$ can be positive: $t_i=t_j$ for some $i,j$, $i\neq j$. This is due to the symmetry of $X$.
We will see that in some cases these $i,j$ are not unique. More precisely, it may happen that there is a set $\{ t_{i_1}, t_{i_2}, \ldots,t_{i_{2a}}\}$ with $i_1, \ldots,i_{2a}$ distinct such that $|\{ t_{i_1}, t_{i_2}, \ldots,t_{i_{2a}}\}|\le a$ is necessary for $f(I,T)>0$. In these cases, the cardinality of
$T$ is less than $k$ provided $f(I,T)>0$.  
\par
We start with the two simplest cases. If $r=1$, we have $f(I,T)>0$ for any $T$, and 
hence $|T|\le k = k-r+1$. Moreover, if $r=k$, we have $f(I,T)>0$ if and only if $t_1=\cdots=t_k$, and hence $|T|=1 = k-r+1$.
\par
Now we assume $1<r<k$.
Our arguments will rely on the proof of Lemma~\ref{prop:PSI}.
Clearly, $1\le g(I)\le k$. From the definition of $\runs$ in the PSA we have $\runs(I)\le k-r$. From \eqref{eq:proofstep1} and \eqref{eq:proofstep2} one infers 
\begin{equation}\label{eq:sglmsoe}
\runs(I)+1\le g(I)\le k-R(I)\,.
\end{equation}
First, we analyze paths $I$ with $S(I)=\emptyset$, or equivalently $\length(S(I))=0$. This implies $\runs(I)=k-r$; otherwise the path-shortening function stops earlier and $S(I)\neq \emptyset$. Therefore, we get from the identity
\begin{equation}\label{eq:identityg}
g(I)=g(S(I)) + \runs(I)\,
\end{equation}
that $g(I)= k-r+1$, which finishes the proof in the case $S(I)=\emptyset$. 
Note that \eqref{eq:identityg} holds for all $I$. This follows from the proof of Lemma~\ref{prop:PSI}.
\par

Next, assume $S(I)\neq \emptyset$. In this case, we immediately see
\begin{equation}\label{eq:lengthsi}
\length(S(I))=k-R(I)-\runs(I)\,.
\end{equation}
Since each element in $S(I)$ has to appear at least twice and $r\ge 2$ we have $\length(S(I))\ge 4$. Moreover, $S(I)$ has $r-R(I)\ge 2$ distinct components.  
As a consequence, it must hold
\begin{equation}\label{eq:boundri}
\runs(I) \le k-r- (r-R(I)) =k-2r+R(I)\,.
\end{equation}
In view of \eqref{eq:identityg}, we have to bound $g(S(I))$. 
Without loss of generality $S(I)$ may be assumed canonical: there is exactly 
one canonical path in the isomorphism class of $S(I)$ and every path in an isomorphism class has the same $g$-value.
If, for example, $S(I)$ happens to be $(3,4,3,4)$, then we will work with the canonical representative $(1,2,1,2)$.
Write $S(I)= (s_1, \ldots, s_{\length(S(I))})$. Since $S(I)$ is canonical, we have
\begin{equation*}
\{s_1, \ldots, s_{\length(S(I))}\} = \{1,\ldots,r-R(I)\}\,.
\end{equation*}
For $ i=1,\ldots,r-R(I)$ define $N_i:= |\{1\le j\le \length(S(I)) : s_j=i\}|$. If we now let $L_i$ be the set of all $u$ such that $(i,t_u)$ appears as an index in $f(S(I),T)$, then $|L_i|=2 N_i$. Finally, define
\begin{equation*}
T_i:= \{t_j\,:\, j\in L_i \} \quad \text{ and } \quad \widetilde{T}_i:= (t_j\,:\, j\in L_i )\,, \quad i=1,\ldots,r-R(I)\,.
\end{equation*}
For example, if $I=(1,2,1,2,3,3)$, then $k=6,r=3$, $N_1=N_2=2$ and we have 
$(S(I),R(I),\runs(I))=((1,2,1,2),1,1)$ and $L_1=L_2=\{1,2,3,4\}$.
\par
By construction, $f(S(I),T)$ can only be positive if $|T_i| \le N_i$. More precisely every $t$-index in the vector $\widetilde{T}_i$ needs to coincide with at least $1$ other $t$-index of this vector. Otherwise,
$\E\big[\prod_{u\in L_i} Y_{i,t_u} \big]=0$
which would imply $f(S(I),T)=0$. The quantity $g(S(I))$ is the maximum number of distinct $t$-indices such that $f(S(I),T)>0$. Hence, there can be at most $\length(\widetilde{T}_i)/2$ distinct $t$-indices in $\widetilde{T}_i$. Since each $t_j$ appears exactly twice in $(\widetilde{T}_1, \ldots, \widetilde{T}_{r-R(I)})$, 
\begin{equation}\label{eq:gsings}
g(S(I))\le 0.5\,\length(S(I))\,.
\end{equation}
\par
Now we are ready to finish the proof of the lemma.
By \eqref{eq:identityg}, \eqref{eq:gsings}, \eqref{eq:lengthsi}, \eqref{eq:boundri},  
in this order, one obtains
\begin{equation*}
\begin{split}
g(I)&= g(S(I))+ \runs(I) \le \frac{\length(S(I))}{2}+ \runs(I)\\
&= \frac{k-R(I)-\runs(I)}{2}+ \runs(I)\\
&\le \frac{k-R(I)+k-2r+R(I)}{2} =k-r\,.
\end{split}
\end{equation*}
\end{proof}
\begin{remark}\label{re:notunique}{\em 
The above proof reveals that $g(I)=k-r+1$ if and only if $S(I)= \emptyset$. For $r$-paths $I$ of length $k$ with $S(I)\neq \emptyset$, the bound $g(I)\le k-r$ is sharp. Consider for instance $I=(1,2,1,2)$, where 
\begin{equation*}
f(I,T)=\E[ Y_{1t_1}Y_{1t_2}Y_{1t_3}Y_{1t_4} ]\, \E[ Y_{2t_1}Y_{2t_2}Y_{2t_3}Y_{2t_4} ]=(\E[ Y_{t_1}Y_{t_2}Y_{t_3}Y_{t_4} ])^2\,.
\end{equation*}
From this relation it is easily deduced that the only canonical representatives $T=(t_1,t_2,t_3,t_4)$ leading to $f(I,T)>0$ are $(1,1,2,2),(1,2,1,2),(1,2,2,1)$ and $(1,1,1,1)$. The first three of them have the highest number of distinct values. We conclude $g(I)=2$. In general, the canonical paths $T$ for which the maximum in \eqref{eq:defmaxg} is attained are not unique, whenever $g(I)\le k-r$. On the other hand, if $g(I)= k-r+1$ there exists exactly one canonical $(k-r+1)$-path $T$ of length $k$ for which the maximum is obtained. This is an immediate consequence of the above proofs. In \cite{bai:silverstein:2010}, Bai and Silverstein present a way to describe this $T$. 
}\end{remark}

For a canonical $r$-path $I$ of length $k$ let 
\begin{equation}\label{eq:defdr}
d(I)= k-r+1-g(I)\,.
\end{equation}
The function $d$ satisfies $0\le d(I) \le k-r$ and $d(S(I))=d(I)$.
The set of canonical $r$-paths of length $k$, denoted by $\mathcal{I}_{r,k}$, can be written as a disjoint union
\begin{equation*}
\mathcal{I}_{r,k}= \bigcup_{u=0}^{k-r} \mathcal{I}_{r,k}(u)\,,
\end{equation*} 
where $\mathcal{I}_{r,k}(u)$ contains those $I$ with $d(I)=u$.
\par
Lemma 3.4 in \cite{bai:silverstein:2010} determines the cardinality of $\mathcal{I}_{r,k}(0)$:
for $k\in \N$ and $r\le k$, 
\beam\label{lem:lemma3.4}
|\mathcal{I}_{r,k}(0)| = \frac{1}{r} \binom{k}{r-1} \binom{k-1}{r-1}\,.
\eeam
\begin{proposition}\label{prop:paths} Assume condition \eqref{eq:assumptionq}. Then the following statements hold for any 
$r$-path $I$ of length $k\ge 1$ and $1\le r\le k$: 
\begin{itemize}
\item[(1)] If $S(I)=\emptyset$, then $F(I)= n^{1-r}$.
\item[(2)] In general, we have 
\begin{equation}\label{eq:proppaths}
F(I)\le 2 \, n^{1-r-d(I)} (2k)^{d(I)} q^{d(I)}\,.
\end{equation}
\end{itemize}
\end{proposition}
\begin{proof}
$S(I)=\emptyset$ is equivalent to $R(I)=r$. By Lemma~\ref{prop:PSI}, 
\begin{equation*}
F(I)=F(S(I)) n^{-R(I)} =F(\emptyset)n^{-r}=n^{1-r}\,.
\end{equation*}
Therefore, we only have to prove \eqref{eq:proppaths} for paths $I$ with $d(I)\ge 1$. Without loss of generality we assume $S(I)$ is a canonical $(r-R(I))$-path. We use the notation for paths with $S(I)\neq \emptyset$ developed in the proof of Lemma~\ref{lem:9.5}. We know that
\begin{equation*}
S(I)= (\pi_1, \ldots, \pi_{\length(S(I))})\,,
\end{equation*}
where $\pi_1, \ldots, \pi_{\length(S(I))}$ is a permutation of the path
\begin{equation*}
I_0 = (\underbrace{1,\ldots,1}_{N_1},\underbrace{2,\ldots,2}_{N_2}, \ldots, \underbrace{r-R(I),\ldots,r-R(I)}_{N_{r-R(I)}})\,.
\end{equation*}
Clearly, $I_0 \in \mathcal{I}_{r-R(I),\length(S(I))}(0)= \mathcal{I}_{r-R(I),k-R(I)-\runs(I)}(0)$. By Lemma~\ref{lem:9.5}, 
\begin{equation*}
g(I_0)=(k-R(I)-\runs(I))-(r-R(I))+1 = k-r-\runs(I)+1
\end{equation*}
and by definition of the function $d(\cdot)$,
\begin{equation*}
g(S(I))=k-r-\runs(I)+1 -d(S(I))=k-r-\runs(I)+1-d(I)\,.
\end{equation*}
\par
The main idea will be to compare $F(S(I))$ to $F(I_0)$. Both of them are sums of expressions of the type
\begin{equation}\label{eq:type}
\prod_{i=1}^{r-R(I)} \E\Big[Y_{i1}^{2m_{i,1}}\cdots Y_{is_i}^{2m_{i,s_i}} \Big] = \prod_{i=1}^{r-R(I)} 
\E\Big[Y_{1}^{2m_{i,1}}\cdots Y_{s_i}^{2m_{i,s_i}} \Big]\,, 
\end{equation}
where for all $i=1,\ldots,r-R(I)$, $1\le s_i\le N_i$, $m_{i,j}\ge 1$ for all $j\ge 1$ and $m_{i,1}+\cdots +m_{i,s_i}=N_i$. We write 
\begin{equation}\label{eq:sdfesofj}
\mathbf{s}=(s_1,\ldots,s_{r-R(I)}) \quad \text{ and } \quad \mathbf{m}_i=(m_{i,1},\ldots,m_{i,s_i})\,,\quad i=1,\ldots,r-R(I)\,.
\end{equation}
Observe that in 
\beao
F(I_0) &=& \sum_{t_1,\ldots,t_{N_1+\cdots+N_{r-R(I)}}=1}^n \E\Big[ Y_{t_{N_1+\cdots+N_{r-R(I)}}} Y_{t_1}^2 
\cdots Y_{t_{N_1-1}}^2Y_{t_{N_1}} \Big] \cdots \\
& &\qquad\qquad \cdots \E\Big[Y_{t_{N_1+\cdots+N_{r-R(I)-1}}} Y_{t_{N_1+\cdots+N_{r-R(I)-1}}+1}^2 \cdots Y_{t_{N_1+\cdots+N_{r-R(I)}}-1}^2Y_{,t_{N_1+\cdots+N_{r-R(I)}}} \Big]
\eeao
the non-zero summands have to satisfy $t_{N_1}=t_{N_2}= \cdots =t_{N_1+\cdots+N_{r-R(I)}}$.  Hence, the above sum is effectively a 
sum only over $g(I_0)$ $t$-indices. The point we want to stress is that there is never a choice, in the sense that even though there are $g(I_0)$ distinct $t$-indices, something like $t_{N_1}=t_1\neq t_2=t_{N_1+N_2}$ is never possible. The reason is that the associated canonical $g(I_0)$-path for the $t$-indices is unique.

For $S(I)\neq \emptyset$, however, the associated canonical $g(S(I))$-path for the $t$-indices is not unique, as mentioned in Remark~\ref{re:notunique}. Depending on the sets $L_i$ there are several possibilities. For instance, for $S(I)=(1,2,1,2)$ we have $L_1=L_2=\{1,2,3,4\}$, $d(I)=1$ and $\length(S(I))=4$. To produce a positive summand one needs $|\{ t_1,t_2,t_3,t_4\}|\le 2$ with every $t$-index appearing at least twice. In this case, $t_1$ has to take the same value as one of the other three $t$-indices. Then there are two $t$-indices left 
which all have to appear at least twice. In this specific example, there are three  canonical paths of $t$-indices 
which are listed in Remark~\ref{re:notunique}. 

We are interested in the general case. How many distinct canonical $g(S(I))$-paths $T$ of length $\length(S(I))$ with $f(S(I),T)>0$ can exist? With the reasoning which lead to \eqref{eq:gsings} one can show that this number is at most $(2d(I)+1)!!$. This bound is attained if $N_1=N_2=\length(S(I))/2$ which implies $L_1=L_2=\{1,\ldots,\length(S(I))\}$. 

For our purpose we will use a much larger bound, namely
\begin{equation}\label{eq:nnnnnnnnn}
(2d(I)+1)!! = (2d(I)+1) (2d(I)-1)\cdots 3 \le (2d(I)+1)^{d(I)} \le (2k)^{d(I)}\,.
\end{equation}

Now let us compare $F(S(I))$ and $F(I_0)$, which look very similar at first sight. The main difference is the dimension of the index sets in the summation. While the sum for $F(I_0)$ contains $n^{g(I_0)}$ positive elements, the sum for $F(S(I))$ has at most $(2k)^{d(I)} n^{g(I_0)-d(I)}$ positive elements. Let $Q_{S(I)}$ denote the set of all canonical $T$ for which the maximum in \eqref{eq:defmaxg} is attained. By the above considerations, $| Q_{S(I)}| \le (2k)^{d(I)}$. Each element in $Q_{S(I)}$ corresponds to a different configuration of $t$-indices in $F(S(I))$, i.e., it tells us which $t$-indices have to be equal. Therefore, we have
\begin{equation}\label{eq:lkhdtoh}
F(S(I))\le \sum_{Q\in Q_{S(I)}} F_Q(S(I))\,,
\end{equation} 
where $F_Q(S(I))$ is defined as follows. Write $Q=(q_1,\ldots, q_{\length(S(I))})$. By construction,\\
 $\{ q_1,\ldots, q_{\length(S(I))}\} = \{1,\ldots, g(S(I))\}$. Set $K_j=\{ 1\le i\le \length(S(I)) :  q_i=j\}$. Then 
\begin{equation}\label{eq:FQ}
F_Q(S(I))= \mathop{\sum_{t_1,\ldots,t_{\length(S(I))}=1}}_{t_l=t_m \, \forall l,m \in K_j\,, 1\le j\le g(S(I))}^n f(S(I),(t_1, \ldots,t_{\length(S(I))}))\,.
\end{equation}
We will show later that 
\begin{equation}\label{eq:assmomentnow}
F_Q(S(I)) \le 2\, q^{d(I)} n^{-d(I)} F(I_0)\,, \quad Q\in Q_{S(I)}\,.
\end{equation}
Then it follows from \eqref{eq:lkhdtoh} and \eqref{eq:assmomentnow} that
\begin{equation*}
\begin{split}
F(S(I))&\le \sum_{Q\in Q_{S(I)}} F_Q(S(I))\\
&\le (2k)^{d(I)} 2\, q^{d(I)} n^{-d(I)} F(I_0) = 2\, (2k)^{d(I)} q^{d(I)} n^{-d(I)} n^{1-r+R(I)}\,.
\end{split}
\end{equation*}
Finally, an application of Lemma~\ref{prop:PSI} gives 
\begin{equation*}
F(I) =n^{-R(I)}F(S(I))\le 2\, (2k)^{d(I)} q^{d(I)} n^{-d(I)} n^{1-r}\,,
\end{equation*}
which completes the proof. 
\par
Next, we show \eqref{eq:assmomentnow} by matching each of the $n^{g(I_0)-d(I)}$ positive summands in \eqref{eq:FQ} with $n^{d(I)}$ of the $n^{g(I_0)}$ positive summands in $F(I_0)$, where we recall that
\begin{equation}\label{eq:FI0}
F(I_0)= \mathop{\sum_{t_1,\ldots,t_{\length(S(I))}=1}}_{t_{N_1}=t_{N_2}= \cdots =t_{N_1+\cdots+N_{r-R(I)}}}^n f(I_0,(t_1, \ldots,t_{\length(S(I))}))\,.
\end{equation}
\par
By {\em matching} we mean the following. Assume we want to prove 
\begin{equation}\label{eq:stggsrbjklfg}
\sum_{i=1}^n A_i \le  \sum_{j=1}^m B_j
\end{equation}
for nonnegative $A_i,B_i$ and $m\ge n$. If for every $i=1,\ldots,n$ there exists a $j_i\in \{1,\ldots, m\}$ such that $A_i\le B_{j_i}$ and the $j_i$'s are distinct, then \eqref{eq:stggsrbjklfg} holds. In this case, we say that each $A_i$ is matched by some $B_{j_i}$.
\par
We say that $f(S(I),(t_1, \ldots,t_{\length(S(I))}))$ and  $f(I_0,(t_1, \ldots,t_{\length(S(I))}))$ are in class $y=\sum_{i=1}^{r-R(I)}s_i$ if they can be written in the form 
\begin{equation}\label{eq:type11}
\prod_{i=1}^{r-R(I)} \E\big[ Y_{1}^{2m_{i,1}}\cdots Y_{s_i}^{2m_{i,s_i}} \big]\,.
\end{equation}
By construction, $y$ takes values in the set $\{ r-R(I), \ldots,\length(S(I)) \}$. A summand in class $y$ is fully determined by the vector $(\mathbf{s},\mathbf{m}_1,\ldots, \mathbf{m}_{r-R(I)})=:(\mathbf{s},\mathbf{m})$; see \eqref{eq:sdfesofj} for this notation.
Hence, we call this summand of type $(\mathbf{s},\mathbf{m})$ and denote it $f_{\mathbf{s},\mathbf{m}}$. Note that the class $y$ is comprised of all elements of type $(\mathbf{s},\mathbf{m})$ such that $\sum_{i=1}^{r-R(I)}s_i=y$ and $\mathbf{m}$ satisfies the restriction stated below equation \eqref{eq:type}.

Let $\mathcal{T}_0(y)$ and $\mathcal{T}_Q(y)$ be index sets which contain the exact type of all  summands (counted with multiplicity) of class $y$ in \eqref{eq:FI0} and \eqref{eq:FQ}, respectively. 
As mentioned before, we must have
\begin{equation*}
\sum_{y=r-R(I)}^{\length(S(I))} |\mathcal{T}_0(y)|=n^{g(I_0)} \quad \text{ and } \quad \sum_{y=r-R(I)}^{\length(S(I))} |\mathcal{T}_Q(y)|=n^{g(I_0)-d(I)}\,.
\end{equation*}
With this notation we can write
\beam\label{eq:rewrfi0}
2 \, F(I_0)&=&2\, \sum_{y=r-R(I)}^{\length(S(I))} \sum_{(\mathbf{s},\mathbf{m})\in \mathcal{T}_0(y)} f_{\mathbf{s},\mathbf{m}}\,,\\\label{eq:rewrfq}
n^{d(I)}\, F_Q(S(I))&=&n^{d(I)}\, \sum_{y=r-R(I)}^{\length(S(I))} \sum_{(\mathbf{s},\mathbf{m})\in \mathcal{T}_Q(y)} f_{\mathbf{s},\mathbf{m}}\,.
\eeam
We show \eqref{eq:assmomentnow} by a {\em matching argument}. 
We start by matching summands in class $\length(S(I))$. From \eqref{eq:type11} we see that elements of class $\length(S(I))$ are necessarily of the form 
\begin{equation*}
\prod_{i=1}^{r-R(I)} \E\big[ Y_{1}^{2}\cdots Y_{N_i}^{2} \big]\,, 
\end{equation*}
in other words they are all equal. Note that 
\begin{equation*}
\begin{split}
|\mathcal{T}_0(\length(S(I)))|&= n (n-1) \cdots (n-g(I_0)+1)\,,\\
|\mathcal{T}_Q(\length(S(I)))|&= n (n-1) \cdots (n-g(I_0)+d(I)+1)\,.
\end{split}
\end{equation*}
Therefore, we have 
\begin{equation*}
n^{d(I)}\, \sum_{(\mathbf{s},\mathbf{m})\in \mathcal{T}_Q(\length(S(I)))} f_{\mathbf{s},\mathbf{m}} \le 2\, \sum_{(\mathbf{s},\mathbf{m})\in \mathcal{T}_0(\length(S(I)))} f_{\mathbf{s},\mathbf{m}}\,.
\end{equation*}
This shows that for each summand in class $\length(S(I))$ on the \lhs~ of \eqref{eq:assmomentnow} we can find at least one summand of the same type on the  \rhs~ of \eqref{eq:assmomentnow}.

Since a large number of summands of the class $\length(S(I))$ have not been used for matching of summands from the same class, we can use \eqref{eq:assumptionq} to match them with summands of classes $\length(S(I))-1, \length(S(I))-2, \ldots, \max(\length(S(I))-d(I),r-R(I))$. 

Applying \eqref{eq:assumptionq} to a summand of class $\length(S(I))-1$ we obtain that it is bounded by $q$ times a summand in class $\length(S(I))$. Hence, we can perform the matching in \eqref{eq:assmomentnow} also between different classes. Clearly, for $y\in\{ r-R(I)+1,\ldots,\length(S(I))\}$ the index set $\mathcal{T}_0(y)$ is much larger than $\mathcal{T}_0(y-1)$. In fact, we have $|\mathcal{T}_0(y)|=n\; c\;|\mathcal{T}_0(y-1)|$ for some constant $c>1$. Note that $y\mapsto |\mathcal{T}_Q(y)|$ is not a strictly increasing function since some $\mathcal{T}_Q(y)$ can be empty. 

The matching is performed as follows: first match the class $\length(S(I))$ summands on the \lhs~ of \eqref{eq:assmomentnow}. Then match the class $\length(S(I))-1$ summands on the \lhs~ of \eqref{eq:assmomentnow} with the remaining class $\length(S(I))$ summands on the \rhs~which have not been used for the matching yet. 

Let $r-R(I)\le u\le \length(S(I))$. The general strategy is to match class $u$ summands on the \lhs~ with class $u,\ldots,\min(u+d(I),\length(S(I)))$ summands on the \rhs. During the matching one tries to use the (still available) class $\min(u+d(I),\length(S(I)))$ summands on the \rhs~ first, then turns to class $\min(u+d(I),\length(S(I)))-1$, and so forth. Whenever a matching between different classes is performed an application of \eqref{eq:assumptionq} is necessary to ensure that the expression on the \lhs~ is bounded by whatever we have matched it with on the \rhs. This leads to powers of $q$ and since $q^{d(I)}$ is the highest possible power we have explained the factor $q^{d(I)}$ in  \eqref{eq:assmomentnow}.
\par
Note that the factor $2$ in  \eqref{eq:assmomentnow} is there to guarantee 
\begin{equation*}
|\mathcal{T}_Q(\length(S(I)))|< 2\, n^{-d(I)} \, |\mathcal{T}_0(\length(S(I)))|\,
\end{equation*}
for sufficiently large $n$, but it is of no central importance. 

The last step in the procedure is the matching of the summands with the highest possible powers on the \lhs~of \eqref{eq:assmomentnow}, which appear when all $t$-indices are equal. They are elements of the class $r-R(I)$. We have
\begin{equation*}
|\mathcal{T}_0(r-R(I))|=|\mathcal{T}_Q(r-R(I))|=n\,,
\end{equation*}
which is a simple explanation why matching of \eqref{eq:rewrfi0} and \eqref{eq:rewrfq} with summands in the same class cannot work in general. Using \eqref{eq:assumptionq} $d(I)$ times, we can bound class $r-R(I)$ summands by class $r-R(I)+d(I)$ summands of which we originally have $|\mathcal{T}_0(r-R(I))|\approx n^{d(I)}|\mathcal{T}_0(r-R(I))|$, which explains the factor $n^{d(I)}$ in \eqref{eq:assmomentnow}. The general matching strategy applies and the proof of \eqref{eq:assmomentnow} is complete.
\end{proof}

\section{Proof of Theorem~\ref{thm:mu1convergence}}\label{sec:5.2}\setcounter{equation}{0}
The following proposition contains our main technical novelty. Its proof is given after the proof of Theorem~\ref{thm:mu1convergence}.
\begin{proposition}\label{thm:limsup}
Assume \eqref{Cpn} and that the iid symmetric field $(X_{it})$ satisfies  \eqref{eq:assumptionq}.
Then the following limit results hold for the largest and smallest eigenvalues $\mu_{(1)}$ and $\mu_{(p)}$ of $\bfR$:
\beam\label{eq:dndfnseg}
\limsup_{\nto} \mu_{(1)} &\le& (1+\sqrt{\gamma})^2 \, \quad \as\\\label{eq:dndfnseg1}
\liminf_{\nto} \mu_{(p)} &\ge &(1-\sqrt{\gamma})^2 \, \quad \as
\eeam
\end{proposition}
\begin{proof}[Proof of Theorem~\ref{thm:mu1convergence}]
(1) If $\E[X^4]<\infty$, \eqref{eq:mlmgd} and \eqref{eq:limitmup} hold for any mean zero distribution as seen in \eqref{eq:drtgdrghdr1}.
\par
\noindent
(2) Now assume \eqref{eq:assumptionq}.
The convergence of $F_{\bfR}$ to a deterministic distribution 
supported on a compact interval implies that the number of the
eigenvalues outside this interval is $o(p)$. Since the right and left endpoints of the \MP law are 
$(1+\sqrt{\gamma})^2$ and $(1-\sqrt{\gamma})^2$, respectively, we conclude from Theorem~\ref{thm:mpcorrelation}(1) that 
\begin{equation*}
\liminf_{n\to \infty}\mu_{(1)} \ge (1+\sqrt{\gamma})^2\quad \as \quad\mbox{and}\quad
\limsup_{n\to \infty}\mu_{(p)} \le (1-\sqrt{\gamma})^2\quad \as \,;
\end{equation*}
see \cite{bai:silverstein:2010} for details. Together with Proposition~\ref{thm:limsup} this completes the proof.
\end{proof}

\subsection*{Proof of equation \eqref{eq:dndfnseg} in Proposition~\ref{thm:limsup}}\label{sec:5.4}

We follow \cite{geman}. By Borel--Cantelli, \eqref{eq:dndfnseg} holds if 
\begin{equation}\label{eq:showsumexp1}
\sum_{n=1}^\infty \E\Big[\Big(\frac{\mu_{(1)}}{z}\Big)^k  \Big] <\infty\,,
\end{equation}
where $z>(1+\sqrt{\gamma})^2$ and $k=k_n\to \infty$. We choose $k$ such that $k/\log n \to \infty$ and $(k^3 q)/n \to 0$, which exists by condition \eqref{eq:assumptionq}. Our goal is to show \eqref{eq:showsumexp1}.
We use that $\E[\mu_{(1)}^k] \le \E[\tr(\bfR)^k]$ and 
\begin{equation*}
\begin{split}
\E[\tr(\bfR)^k]&= \sum_{i_1,\ldots,i_k=1}^p \sum_{t_1,\ldots,t_k=1}^n \E[ Y_{i_1t_k} Y_{i_1t_1} Y_{i_2t_1}Y_{i_2t_2}Y_{i_3t_2} Y_{i_3t_3} \cdots Y_{i_kt_{k-1}} Y_{i_kt_k}  ]\\
&= \sum_{i_1,\ldots,i_k=1}^p F(i_1,\ldots,i_k)\, .
\end{split}
\end{equation*}
We rewrite $\E[\tr(\bfR)^k]$ by sorting according to the number of distinct components in the path $(i_1, \ldots, i_k)$. Any $r$-path of length $k$ is an element in the disjoint union $\mathcal{J}_{r,k}(0) \cup \cdots \cup \mathcal{J}_{r,k}(k-r)$, where $\mathcal{J}_{r,k}(u)$ is the set of all $r$-paths $I$ of length $k$ with $d(I)=u$; see \eqref{eq:defdr} for the definition of $d(I)$. Hence we have
\begin{equation}\label{eq:disjointkds}
\{1,\ldots, p\}^k = \bigcup_{r=1}^k \bigcup_{u=0}^{k-r} \mathcal{J}_{r,k}(u)\,.
\end{equation}

Given a path $I\in \mathcal{J}_{r,k}(u)$ we can look at the positions where the $r$ distinct components appear for the first time. There are $r$ such positions. The first such position is always $1$, in general $i_1$ can take $p$ different values. For the second such position there are $(p-1)$ possibilities; the original $p$ minus the one from the first position. In total there are $p(p-1)\cdots (p-r+1)$ ways to assign values to these $r$ positions. For this reason
\begin{equation}\label{eq:slktej}
|\mathcal{J}_{r,k}(u)| = p(p-1)\cdots (p-r+1) |\mathcal{I}_{r,k}(u)|\,,
\end{equation}
where $\mathcal{I}_{r,k}(u)$ is the set of all canonical $r$-paths $I$ of length $k$ with $d(I)=u$. The only difference between the definitions of $\mathcal{J}_{r,k}(u)$ and $\mathcal{I}_{r,k}(u)$ is that the elements of the latter are canonical. 
Note that $\mathcal{I}_{k,k}(u)=\emptyset$ for all $u\ge 1$.
\par
In view of \eqref{eq:disjointkds} and \eqref{eq:slktej} we obtain
\beam\label{eq:oklj}
\E[\tr(\bfR)^k]&=& \sum_{r=1}^k \sum_{u=0}^{k-r} \sum_{I\in \mathcal{J}_{r,k}(u)} F(I)\nonumber\\
&=& \sum_{r=1}^k p(p-1)\cdots (p-r+1) \sum_{u=0}^{k-r} \sum_{I\in \mathcal{I}_{r,k}(u)} F(I)\label{eq:xx}\\
&\le& \sum_{r=1}^k p^r  \sum_{I\in \mathcal{I}_{r,k}(0)} F(I)
+\sum_{r=1}^{k-1} p^r \sum_{u=1}^{k-r} |\mathcal{I}_{r,k}(u)| \max_{I\in \mathcal{I}_{r,k}(u)}F(I)=:S_1+S_2\,.\nonumber
\eeam
By Proposition~\ref{prop:paths}, \eqref{lem:lemma3.4} and since $|\mathcal{I}_{r,k}(0)|\le \binom{k-1}{r-1}^2$, we have
\begin{equation}\label{eq:sumS1}
S_1 \le \sum_{r=1}^k p^r \binom{k-1}{r-1}^2 n^{1-r} = p \sum_{r=1}^k \binom{k-1}{r-1}^2 \Big(\frac{p}{n}\Big)^{r-1}\,.
\end{equation}

Next we bound $S_2$. Consider $1\le u\le k-r$. We will see how elements of $\mathcal{I}_{r,k}(u)$ can be constructed by modifying elements of $\mathcal{I}_{r,k}(0)$. Let $I\in \mathcal{I}_{r,k}(u)(N_1, \ldots, N_r)$ be the subset of $\mathcal{I}_{r,k}(u)$ for whose elements the integer $i$ appears exactly $N_i$ times as a component. Here $N_i, i=1,\ldots,r$ are positive integers satisfying $N_1+\cdots+N_r=k$. Obviously it is possible to obtain $I$ by permuting the components of any $I_0 \in \mathcal{I}_{r,k}(0)(N_1, \ldots, N_r)$. Consider the following permutation of $I_0$: two components of $I_0$ exchange places, all others remain untouched. We denote such a {\em switching permutation} by $SP$. The number of such permutations is bounded by $k^2/2$. Indeed, the first component can switch places with the remaining $k-1$ components, the second with $k-2$ components, etc. In total there are 
\begin{equation*}
(k-1)+(k-2)+ \cdots+ 1 = \sum_{j=1}^{k-1} j = \frac{(k-1)k}{2} \le \frac{k^2}{2}
\end{equation*}
ways how two components can switch positions. 

Let $u=1$. For any $I\in \mathcal{I}_{r,k}(u)(N_1, \ldots, N_r)$ there exists an $I_0 \in \mathcal{I}_{r,k}(0)(N_1, \ldots, N_r)$ and a switching permutation $SP$ such that $I= SP(I_0)$. Here $SP$ and $I_0$ are in general not unique. This is a consequence of the proof of Lemma~\ref{lem:9.5}. This implies
\beao
|\mathcal{I}_{r,k}(u)| \le |\mathcal{I}_{r,k}(0)| \frac{k^2}{2}\,.
\eeao
Similarly, for $1\le u\le k-r$ and $I\in \mathcal{I}_{r,k}(u)(N_1, \ldots, N_r)$ there exists an $I_0 \in \mathcal{I}_{r,k}(0)(N_1, \ldots, N_r)$ and switching permutations $SP_1,\ldots, SP_u$ such that
$
I= SP_1 \circ \cdots \circ SP_u (I_0)\,,
$
which shows
\begin{equation}\label{eq:lsetseet}
|\mathcal{I}_{r,k}(u)| \le |\mathcal{I}_{r,k}(0)| \Big(\frac{k^2}{2}\Big)^u\,.
\end{equation}
\par
Now we are ready to bound $S_2$. From Proposition~\ref{prop:paths} we get 
\begin{equation*}
\max_{I\in \mathcal{I}_{r,k}(u)}F(I) \le 2 n^{1-r-u}(2k)^u q^u 
\end{equation*}
and therefore,
\begin{equation*}
\begin{split}
S_2&\le \sum_{r=1}^{k-1} p^r \sum_{u=1}^{k-r} \binom{k-1}{r-1}^2 \Big(\frac{k^2}{2}\Big)^u 2 n^{1-r-u}(2k)^u q^u\\
&= p \sum_{r=1}^{k-1} \binom{k-1}{r-1}^2 \Big(\frac{p}{n}\Big)^{r-1} 2 \sum_{u=1}^{k-r} \Big( \frac{k^3 q}{n} \Big)^u\\
&\le p \sum_{r=1}^{k-1} \binom{k-1}{r-1}^2 \Big(\frac{p}{n}\Big)^{r-1} 2 \Big[\Big(1- \frac{k^3 q}{n} \Big)^{-1}-1\Big]\,.
\end{split}
\end{equation*}
Finally, we have the bound
\begin{equation*}
\begin{split}
\E[\tr(\bfR)^k]&\le S_1+S_2\le p \sum_{r=1}^k \binom{k-1}{r-1}^2 \Big(\frac{p}{n}\Big)^{r-1} \Big( 1+ 2 \Big[\Big(1- \frac{k^3 q}{n} \Big)^{-1}-1\Big] \1_{\{r<k\}} \Big)\\
&\le p \sum_{r=1}^k \binom{2k-2}{2r-2} \Big(\frac{p}{n}\Big)^{r-1} \Big( 1+ 2 \Big[\Big(1- \frac{k^3 q}{n} \Big)^{-1}-1\Big]\1_{\{r<k\}}\Big)\\
&\le p\sum_{r=0}^{2k-2} \binom{2k-2}{r} \Big(\sqrt{\frac{p}{n}}\Big)^{r} \Big( 2 \Big(1- \frac{k^3 q}{n} \Big)^{-1}-1 \Big)^{2k-2-r}\\
&= \Big[p^{1/(k-1)} \Big( 2 \Big(1- \frac{k^3 q}{n} \Big)^{-1}-1 +\sqrt{\frac{p}{n}}\Big)^2\Big]^{k-1} \le \eta^k\,,
\end{split}
\end{equation*}
where $\eta$ is a constant satisfying $(1+\sqrt{\gamma})^2<\eta<z$. The last inequality follows from $p^{1/(k-1)} \to 1$ and
\begin{equation*}
\lim_{\nto} \Big( 2 \Big(1- \frac{k^3 q}{n} \Big)^{-1}-1 +\sqrt{\frac{p}{n}}\Big)^2 =(1+\sqrt{\gamma})^2\,.
\end{equation*}
This shows \eqref{eq:showsumexp1} which concludes the proof.


\subsection{Proof of equation \eqref{eq:dndfnseg1} in Proposition~\ref{thm:limsup}}\label{sec:5.6}
We start with the following result.
\begin{proposition}\label{prop:unified}
Assume \eqref{Cpn}. If the iid entries $(X_{it})$ are symmetric and satisfy condition \eqref{eq:assumptionq} then
\beam\label{eq:jo1}
\limsup_{n\to \infty} \twonorm{\bfR-(1+\gamma) \bfI} \le 2 \sqrt{\gamma}\,\quad \as
\eeam
\end{proposition}

\begin{proof}
The general idea is the same as in the proof of equation \eqref{eq:dndfnseg} in Proposition~\ref{thm:limsup}: we will 
bound the spectral norm of $\bfR-(1+\gamma) \bfI$ by the trace of high powers of this matrix and then take an appropriate root.
To this end we choose an integer \seq\ $k=k_n\to \infty$ such that $k/\log n \to \infty$ and $(k^3 q)/n \to 0$, which exists by condition \eqref{eq:assumptionq}. 
Since the matrices $\bfR$ and $(1+\gamma) \bfI$ commute we have
\begin{equation*}
(\bfR-(1+\gamma) \bfI)^{2k}= \sum_{i=0}^{2k} \binom{2k}{i} \bfR^i (-1)^i (1+\gamma)^{2k-i} \bfI\,.
\end{equation*}
By linearity of the trace,
\begin{equation}\label{eq:commutegh}
\E[ \tr(\bfR-(1+\gamma) \bfI)^{2k}]= p\,(1+\gamma)^{2k} \Big[1 + p^{-1} \sum_{i=1}^{2k} \binom{2k}{i}  \Big(\frac{-1}{1+\gamma}\Big)^i \E[\tr \bfR^i]\Big]\,.
\end{equation}
From \eqref{eq:oklj} combined with \eqref{lem:lemma3.4} we know that for $n$ sufficiently large
\begin{equation}\label{eq:gsegkasdf1}
\E[\tr \bfR^i]\ge p \sum_{r=1}^i \frac{(p-1)(p-2)\cdots (p-r+1)}{n^{r-1}} \frac{1}{r}\binom{i}{r-1}\binom{i-1}{r-1} =p\, \beta_i(\gamma)\,(1-\delta_n)\,,
\end{equation}
where $\delta_n=O(1/n)$. Additionally, we established
\begin{equation}\label{eq:gsegkasdf2}
\E[\tr \bfR^i]\le p\, \beta_i(\gamma)\,\Big(1+\frac{2k^3 q}{n}\Big) (1+\delta_n)\,.
\end{equation}
Hence, by \eqref{eq:commutegh}, Lemma~\ref{lem:fkgamma}, and noting that $f_k$ is continuous on $\R$, and $p/n\to \gamma\in(0,1]$, 
we have for $n$ sufficiently large,
\begin{equation*}
\begin{split}
\E[ \tr(\bfR-(1+\gamma) \bfI)^{2k}]
& = p(1+\gamma)^{2k} \Big[ 1 +\sum_{i=1}^{2k} \binom{2k}{i}  \Big(\frac{-1}{1+\gamma}\Big)^i  \,\beta_i(\gamma)\Big]
\big(1+O(2k^3 q_n/n)\big)\\
&= p(1+\gamma)^{2k} f_k(\gamma) \big(1+O(2k^3 q_n/n)\big)\\
&\le p\,(1+\gamma) (4\gamma)^k \big(1+O(2k^3 q_n/n)\big) < z^{2k}\,,
\end{split}
\end{equation*}
for any $z>2\sqrt{\gamma}$. The last inequality follows from
\begin{equation*}
\lim_{\nto} p^{1/(2k)} \big(1+2k^3 q_n/n \big)^{1/(2k)}(1+\gamma)^{1/(2k)} 
= 1\,.
\end{equation*}
Using the same Borel-Cantelli argument as in the proof of \eqref{eq:dndfnseg}, one obtains the desired relation
\begin{equation*}
\limsup_{n\to \infty} \twonorm{\bfR-(1+\gamma) \bfI} \le  2 \sqrt{\gamma}\quad \as
\end{equation*}
\end{proof}
With Proposition~\ref{prop:unified} we can finish the proof of \eqref{eq:dndfnseg1}.
We have 
\begin{equation*}
\twonorm{\bfR-(1+\gamma) \bfI} = \max \{\mu_{(1)}- (1+\gamma), -\mu_{(p)}+(1+\gamma)\}\,.
\end{equation*}
From \eqref{eq:jo1} we conclude
\begin{equation*}
\begin{split}
\limsup_{n\to \infty}\mu_{(1)} &\le \phantom{-} 2 \sqrt{\gamma}+1+\gamma=(1+\sqrt{\gamma})^2 \quad \as\,,\\
\liminf_{n\to \infty}\mu_{(p)} &\ge -2 \sqrt{\gamma}+1+\gamma=(1-\sqrt{\gamma})^2\quad \as
\end{split}
\end{equation*}

\section{Proof of Theorem~\ref{thm:mpcorrelation}}\label{sec:6}\setcounter{equation}{0}
\subsection{Proof of Theorem~\ref{thm:mpcorrelation}(1)}\label{sec:5.1}
We appeal to the proof of Theorem 2.3 in \cite{bai:zhou:2008}.
The following lemma is a version of Corollary 1.1 in \cite{bai:zhou:2008}.
\begin{lemma}\label{lem:baicorrected}
Let 
$\mathbf{B} =\mathbf{B}_n= (B_{jk})$ be a non-random $n\times n$ matrix with bounded norm
and 
\beao
\mathcal{S}=\{ (i_1,j_1,i_2,j_2): 1\le i_1,j_1,i_2,j_2\le n\}\backslash \{(i_1,j_1,i_2,j_2): i_1=i_2 , j_1=j_2 \text{ or } i_1=j_2\neq i_2=j_1  \}\,.
\eeao
If $\E[Y^4 ] = o(n^{-1})$,
\beam\label{eq:con1}
n\,\var(Y_{1}Y_{2}) &\to& 0\,,\\
\label{eq:con2}
V_n=n^2\, \sum_{\mathcal{S}} \big(\cov(Y_{i_1}Y_{j_1},Y_{i_2}Y_{j_2})\big)^2 &\to& 0\,, 
\eeam
then Condition 1 of Theorem 1.1 in \cite{bai:zhou:2008} holds, i.e.,
\begin{equation*}
\E[ |\Y_{0} \mathbf{B} \Y_{0}'- \tr(\mathbf{B}\,\E[\bfY_0\bfY_0'])|^2] = o(1),
\end{equation*}
where $\Y_{0}= (Y_{1}, \ldots, Y_{n})$.
\end{lemma}
\begin{proof}
We have for some constant $c>0$, 
\beao
\lefteqn{\E[ |\Y_{0} \mathbf{B} \Y_{0}'- \tr(\mathbf{B}\E[\bfY_0\bfY_0'])|^2]}\\ &=& 
\E\Big[ \Big| \sum_{i_1,j_1=1}^n B_{i_1j_1} (Y_{i_1} Y_{j_1} - \E[Y_{i_1} Y_{j_1}]) \Big|^2\Big]\\
&=& \sum_{i_1,j_1=1}^n \sum_{i_2,j_2=1}^n B_{i_1j_1}B_{i_2j_2}
\cov (Y_{i_1} Y_{j_1},Y_{i_2} Y_{j_2})\\
&\le &c\, \big[n  \,\var(Y_{1}^2)
 + n \, \var(Y_{11} Y_{12})\big] 
+ \sum_{\mathcal{S}} B_{i_1j_1}B_{i_2j_2} \cov(Y_{i_1} Y_{j_1}, Y_{i_2} Y_{j_2})\,.
\eeao
By assumption,
$n  \,\var(Y^2)  = n \,(\E[Y^4]-n^{-2})\to 0$. The second summand converges to zero by \eqref{eq:con1}. 
It is shown in 
\cite{bai:zhou:2008} that the last summand is bounded by
\begin{equation*}
c\, n \Big(\sum_{\mathcal{S}} \big( \cov(Y_{i_1} Y_{j_1},Y_{1i_2} Y_{1j_2})\big)^2 \Big)^{1/2}
\end{equation*}
which converges to zero by \eqref{eq:con2}.
\end{proof}
\begin{remark}{\em Lemma~\ref{lem:baicorrected} corrects the proof of Theorem 2.3 and Corollary 1.1 in \cite{bai:zhou:2008}. 
In the latter paper
it is claimed that
\beao
V_n'=n^2\sum_{\mathcal{S'}} \big(\cov(Y_{i_1}Y_{j_1},Y_{i_2}Y_{j_2})\big)^2 \to 0\,,
\eeao 
where 
\beao
\mathcal{S'}=\{ (i_1,j_1,i_2,j_2): 1\le i_1,j_1,i_2,j_2\le n\}\backslash \{(i_1,j_1,i_2,j_2): i_1=i_2 \neq j_1=j_2 \text{ or } i_1=j_2\neq i_2=j_1  \}\,. 
\eeao
However, $\mathcal{S}'$ contains the quadruples $(i,i,i,i)$. Hence
\beao
V_n'&\ge&  n p^2 \big( \var(Y^2)  \big)^2= n^{-1} \,p^2\, ( n\,\E[Y^4 ] )^2 - 2\, \frac{p^2}{n^2}\, (n\, \E[Y^4 ]) + \frac{p^2}{n^3}\,,
\eeao
which does not necessarily converge to zero since 
$n\,\E[Y^4 ]$ may converge to zero arbitrarily slowly. }
\end{remark}
\par
Now we are ready for the proof of Theorem \ref{thm:mpcorrelation}(1).
If the \ds\ of $X$ is in the domain  of attraction of the normal law
the claim follows from Theorem 2.3 in \cite{bai:zhou:2008}, using our Lemma~\ref{lem:baicorrected}.
\par
Now assume the alternative condition \eqref{eq:condX}.
We will apply Theorem 2.2 in \cite{bai:zhou:2008} and our Lemma~\ref{lem:baicorrected}. 
Our goal is to find the limiting spectral distribution of $\bfR=\Y\Y'$ via the limit of the 
Stieltjes transform of $\Y'\Y$, using the fact that
$\Y\Y'$ and $\Y'\Y$ have the same non-zero eigenvalues. 
Since $\lambda_{(i)}=0$ for any of these matrices whenever $i>n\vee p$ 
we obtain a connection between the two spectral distributions:
\beao
F_{\Y'\Y}=\Big(1-\frac{p}{n}\Big) \1_{[0,\infty)} + \frac{p}{n} F_{\Y\Y'}\,.
\eeao
Hence 
\beam\label{eq:sgdsglop}
s_{\bfR}(z)&=& \int \frac{1}{x-z} \dint F_{\bfR}(x)\nonumber\\
&=& \int \frac{1}{x-z} \dint \Big(\frac{n}{p} F_{\Y'\Y} - \Big( \frac{n}{p}-1 \Big) \1_{[0,\infty)} \Big)(x)\nonumber\\
&=& \frac{n}{p} s_{\Y'\Y}(z)- \Big( \frac{n}{p}-1 \Big) \frac{1}{-z}\,,\quad z\in\C^+\,,
\eeam
where we used that for a constant $c\neq 0$ we have $s_{c\bfA}(z)=c^{-1} s_{\bfA}(cz)$.
\par
We introduce the $n\times n$ matrix $\bfT=(T_{ij})=(p\,\E[Y_{i}Y_{j}])$ which is a circulant matrix whose eigenvalues can be determined as $T_{11}+(n-1) T_{12}$ and $T_{11}-T_{12}$ where the latter appears with multiplicity $n-1$. By \cite{gine:goetze:mason:1997}, we have $T_{ij}=O(n^{-1})$ for $i\neq j$ and hence $\twonorm{\bfT}$ is bounded. The empirical spectral distribution 
\begin{equation*}
F_{\bfT}(x) = \frac{1}{n} \sum_{j=1}^n \1_{\{\lambda_j(\bfT)\le x\}}= \frac{\1_{\{T_{11}+(n-1) T_{12}\le x\}}}{n} +\frac{n-1}{n} \1_{\{T_{11}-T_{12} \le x\}}
\end{equation*}
converges to the degenerate distribution $H_\gamma$ with all mass at $\lim_{\nto} (T_{11}-T_{12})= \lim_{\nto}p/n =\gamma$.
\par
Next we verify the assumptions of Lemma~\ref{lem:baicorrected}.
We have 
\beao
n\, \E[\var(Y_1Y_2)]&=&  n\,\big(\E[(Y_1 Y_2)^2] - (\E[Y_{1}Y_{2}])^2 \big)\\
&\le& n \Big(\frac1 {n(n-1)} - o(n^{-2})\Big) \to 0\,, \qquad \nto\,.
\eeao
This implies \eqref{eq:con1}.
\par Now we turn to $V_n$ in \eqref{eq:con2}.
If we distinguish  between the types of indices in $\mathcal{S}$ we find that either possible structure 
for the summands $(Y_{i_1}Y_{j_1}-\E[Y_{i_1}Y_{j_1}])( Y_{i_2}Y_{j_2}-\E[Y_{i_2}Y_{j_2}])$ in $V_n$
is of the type $Y_{1}^3Y_{2}, Y_{1}^2Y_{2}Y_{3}$ or $Y_{1}Y_{2}Y_{3}Y_{4}$. Keeping this in mind, we conclude that for some constant $c>0$,
\beao
V_n&\le& c\,\Big(n^4\,\big( \cov(Y_1^2,Y_1Y_2)\big)^2 + n^5  \,\big( \cov(Y_1^2,Y_2Y_3)\big)^2
+ n^5 \, \big( \cov(Y_1Y_2,Y_2Y_3)\big)^2\\
&&+n^6\,\big(\cov(Y_1Y_2,Y_3Y_4)\big)^2\Big)\\ 
&\le & c \,\Big( n^4\,\big(\E[Y_{1}^3 Y_{2}] -(1/n) \E[Y_{1} Y_{2}] \big)^2 
+ n^5 \, \big( \E[Y_{1}^2 Y_{2}Y_{3}] -(1/n)  \E[Y_{1} Y_{2}] \big)^2\\
&&+ n^5\, \big(  \E[Y_{1} Y_{2}^2 Y_{3}] -  (\E[Y_{1} Y_{2}])^2 \big)^2+  n^6\,\big( \E[Y_{1} Y_{2}Y_{3}Y_{4}] -  
(\E[Y_{1} Y_{2}])^2 \big)^2\Big)\,.
\eeao
The \rhs\ converges to zero in view of assumption \eqref{eq:condX} and because (see \cite{gine:goetze:mason:1997}) 
\begin{equation*}
\E[Y_{1}^3 Y_{2}] \le \E[Y_{1} Y_{2}], \quad \E[Y_{1}^2 Y_{2}Y_{3}] = O(n^{-3})\quad \text{ and } \quad 
\E[Y_{1} Y_{2}Y_{3}Y_{4}] = O(n^{-4})\,.
\end{equation*}
Applications of Theorem 2.2 in \cite{bai:zhou:2008} and our Lemma~\ref{lem:baicorrected} yield
for $s=\lim_{\nto}s_{\bfY'\bfY}$, 
\begin{equation*}
\begin{split}
s(z)&= \int \frac{1}{\omega (1-\gamma^{-1}-\gamma^{-1}zs(z))-z} \dint H_\gamma(\omega)\\
&= \frac{1}{\gamma (1-\gamma^{-1}-\gamma^{-1}zs(z))-z}\,,
\end{split}
\end{equation*}
Thus $s=s(z)$ is the solution of the quadratic equation
\begin{equation*}
s^2 z + s (1+z-\gamma) +1=0.
\end{equation*}
By convention of \cite{bai:silverstein:2010}, the square root of 
a complex number is the one with a positive imaginary part. Hence
\beao
s(z)= \frac{-(\gamma^{-1}z+\gamma^{-1}-1)+\sqrt{(\gamma^{-1}z-\gamma^{-1}-1)^2-4\gamma^{-1}}}{2\gamma^{-1}z}.
\eeao
Writing $m$ for the limiting Stieltjes transform of $F_{\Y\Y'}$, we conclude from \eqref{eq:sgdsglop} and since $n/p\to \gamma^{-1}$ that
\beao
m(z)&=&\gamma^{-1} s(z)+\frac{\gamma^{-1}-1}{z}\\
&=&\frac{1-\gamma -z +\sqrt{(1+\gamma-z)^2-4\gamma}}{2 \gamma z}\,,
\eeao
which we recognize as the Stieltjes transform of the \MP law in \eqref{eq:MP}; see \eqref{eq:stieltjestransform}. The proof is complete.

\subsection{Proof of Theorem~\ref{thm:mpcorrelation}(2)}\label{sec:5.5}
Assume $\liminf_{\nto} n\,\E [ Y^4] =\delta >0$.
For $k\ge 1$, the expected moments of the empirical spectral distribution $F_\bfR$ are 
\begin{equation}\label{eq:expmomentsr}
\wt \beta_k = \E\Big[\int x^k \dint F_{\bfR}(x)\Big]=p^{-1}\E[\tr\,\bfR^k]=p^{-1}\sum_{i_1,\ldots,i_k=1}^p F(i_1,\ldots,i_k)\,.
\end{equation}
From \eqref{eq:xx} we know that 
\begin{equation*}
\begin{split}
p^{-1}\E[\tr(\bfR)^k]
&\ge \sum_{r=1}^k (p-1)(p-2)\cdots (p-r+1) \Big(\sum_{I\in \mathcal{I}_{r,k}(0)} +\sum_{I\in \mathcal{I}_{r,k}(1)}\Big)
F(I)  =: S_3+S_4\,.
\end{split}
\end{equation*}
By Proposition~\ref{prop:paths} and \eqref{lem:lemma3.4}, we have
\begin{equation}\label{eq:sumS11}
\lim_{\nto} S_3 = \sum_{r=1}^{k} \frac{1}{r} \binom{k}{r-1}\binom{k-1}{r-1}\gamma^{r-1}=\beta_k(\gamma)\,,
\end{equation}
which we recognize from \eqref{eq:momentsmp} as the $k$-th moment of the \MP law. 
\par
Next, observe that for $k\ge 4$ and $2\le r\le k-2$, $\mathcal{I}_{r,k}(1)$ contains the element 
\begin{equation*}
I_r = (1,2,1,2,\underbrace{2,\ldots,2}_{k-r-2},3,\ldots,r)\,.
\end{equation*}
One checks that $R(I_r)=r-2$ and $S(I_r)=(1,2,1,2)$; consult the PSA and Definition~\ref{def:pathshortening} for the definitions of $R(\cdot)$ and $S(\cdot)$. Moreover, by symmetry of $Y_{it}$
we have	
\begin{equation*}
\begin{split}
F(1,2,1,2)&=
\sum_{t_1,\ldots,t_4=1}^n \E[ Y_{1t_1}Y_{1t_2}Y_{1t_3}Y_{1t_4} 
Y_{2t_1}Y_{2t_2}Y_{2t_3}Y_{2t_4}]= \sum_{t_1,\ldots,t_4=1}^n (\E[ Y_{t_1}Y_{t_2}Y_{t_3}Y_{t_4} ])^2\\
&= \sum_{t_1=1}^n (\E[ Y_{t_1}^4])^2 + 3 \sum_{t_1\neq t_2=1}^n (\E[ Y_{t_1}^2Y_{t_2}^2 ])^2\ge \frac{1}{n} (n\E[ Y^4])^2\,.
\end{split}
\end{equation*}
By Lemma~\ref{prop:PSI} we have 
\begin{equation*}
F(I_r)= n^{2-r} F(1,2,1,2)\ge n^{1-r}(n\E[ Y^4])^2
\end{equation*}and consequently
\begin{equation*}
\begin{split}
\liminf_{\nto} S_4&\ge \liminf_{\nto} \sum_{r=2}^{k-2} (p-1)(p-2)\cdots (p-r+1)  F(I_r)\\
&\ge \liminf_{\nto} \sum_{r=2}^{k-2} (p-1)(p-2)\cdots (p-r+1)  n^{1-r}(n\E[ Y^4])^2=  \delta^2 \sum_{r=2}^{k-2} \gamma^{r-1}\,.
\end{split}
\end{equation*}
This together with \eqref{eq:sumS11} proves $\liminf_{\nto} \wt \beta_k>\beta_k(\gamma)$, as desired.

\appendix
\section{}\setcounter{equation}{0}
In this section we provide some auxiliary tools for the proofs of the main results.
\begin{lemma}\label{lem:fkgammaneww}
Let $k\in \N$ and $1\le j\le k$. Then
\begin{equation}\label{eq:vdsvsg}
-\sum_{i=2j-1}^{2k} (-1)^i \binom{2k}{i} \binom{i-1}{2j-2}=1
\end{equation}
\end{lemma}
\begin{proof}
For $1\le j\le k$ we rewrite \eqref{eq:vdsvsg} as 
\begin{equation}\label{eq:vdsvsg2}
-\sum_{i=2j-1}^{2k} (-1)^i \binom{2k}{i} \frac{(i-1)!}{(i+1-2j)!}=(2j-2)!\,.
\end{equation}
We define the functions
\begin{equation*}
u(x)=(x-1)^{2k}-1\,,\quad v(x)=\sum_{i=1}^{2k}  \binom{2k}{i}(-1)^i x^{i-1}\,,\quad \text{and } \quad w(x)=\frac{1}{x}\,.
\end{equation*}
Then $v(x)=u(x)w(x)$ and since
\begin{equation*}
\frac{(i-1)!}{(i+1-2j)!}=(i-1)(i-2)\cdots (i-2j+2)\,,
\end{equation*}
equation \eqref{eq:vdsvsg2} is equivalent to an equation for the $(2j-2)$-th derivative of $v$ evaluated at $1$,
\begin{equation*}
v^{(2j-2)}(1)=-(2j-2)!\,.
\end{equation*}
By Leibniz's rule for differentiation, one gets
\begin{equation*}
v^{(2j-2)}(x)= (uw)^{(2j-2)}(x)= \sum_{\ell=0}^{2j-2} \binom{2j-2}{\ell} u^{(\ell)}(x) w^{(2j-2-\ell)}(x)\,.
\end{equation*}
Observe that $u^{(0)}(1)=-1$ and $u^{(\ell)}(1)=0$ for $1\le \ell \le 2j-2$. Furthermore, $w^{(2j-2-\ell)}(1) = (2j-2-\ell)!$. Hence, we conclude
\begin{equation*}
v^{(2j-2)}(1)= \sum_{\ell=0}^{2j-2} \binom{2j-2}{\ell} u^{(\ell)}(1) w^{(2j-2-\ell)}(1)=- (2j-2)!\,,
\end{equation*}
completing the proof.
\end{proof}
\par
For $k\in \N$ and $x\in [0,1]$, define the function
\begin{equation}\label{eq:fkgamma}
f_k(x)=1+\sum_{i=1}^{2k} \binom{2k}{i} \Big( \frac{-1}{1+x}\Big)^i \sum_{r=1}^i \frac{1}{r} \binom{i}{r-1} \binom{i-1}{r-1} x^{r-1}\,.
\end{equation}
The following is our key lemma.

\begin{lemma}\label{lem:fkgamma}
We have for $k\in \N$ and $x\in [0,1]$
\begin{equation*}
f_k(x) = 1 - \sum_{j=1}^k \frac{x^{j-1}}{(1+x)^{2j-1}} \frac{(2j-2)!}{j! (j-1)!} \le \frac{ (4x)^k}{(1+x)^{2k-1}}\,.
\end{equation*}
\end{lemma}
\begin{proof}
From \cite[page 41]{bai:silverstein:2010} we know that 
\begin{equation*}
\sum_{r=1}^i \frac{1}{r} \binom{i}{r-1} \binom{i-1}{r-1} x^{r-1}
= \sum_{r=0}^{\lfloor(i-1)/2 \rfloor} x^r (1+x)^{i-1-2r} \frac{(i-1)!}{(i-1-2r)! r! (r+1)!}\,.
\end{equation*}
Changing the order of summation one obtains
\begin{equation*}
\begin{split}
f_k(x)-1&= \sum_{r=0}^{k-1} \sum_{i=2r+1}^{2k} \binom{2k}{i} (-1)^i x^r (1+x)^{-1-2r}\frac{(i-1)!}{(i-1-2r)! r! (r+1)!}\\
&= \sum_{j=1}^k \frac{x^{j-1}}{(1+x)^{2j-1}} \frac{1}{j! (j-1)!} \sum_{i=2j-1}^{2k} \binom{2k}{i} (-1)^i \frac{(i-1)!}{(i+1-2j)!}\\
&= -\sum_{j=1}^k \frac{x^{j-1}}{(1+x)^{2j-1}} \frac{(2j-2)!}{j! (j-1)!}\,,
\end{split}
\end{equation*}
where the last equality followed from Lemma~\ref{lem:fkgammaneww} and its equivalent formulation \eqref{eq:vdsvsg2}.

For $j \in \N$ define
\begin{equation}\label{eq:gkgamma}
g_j(x)= \frac{(1+x)^{2j-1}}{x^j} f_j(x)\,.
\end{equation}
We have $g_1(x)=1$ and $g_2(x)=2+x$. A straightforward induction proves the recursion
\begin{equation}\label{eq:gkgammarec}
g_j(x)= \frac{(1+x)^2 g_{j-1}(x)-g_{j-1}(0)}{x}\,, \quad j\ge 2\,.
\end{equation}
From this recursive construction one deduces that $g_j(x)$ is a polynomial of degree $j-1$ with positive coefficients.

Next we show $g_k(x)\le 4^k$. Clearly we have $g_1(x)\le 4$ and $g_2(x)\le 4^2$. Therefore assume
\begin{equation*}
\| g_{k-1}\|_{[0,1]} := \sup_{y\in [0,1]} |g_{k-1}(y)| \le 4^{k-1}\,.
\end{equation*}
Then for $x\in [0,1]$,
\begin{equation*}
\begin{split}
g_k(x) &= \frac{x (2+x) g_{k-1}(x)+g_{k-1}(x)-g_{k-1}(0)}{x}\\
&\le (2+x) g_{k-1}(x)+g_{k-1}(1) \le (2+x+1) \| g_{k-1}\|_{[0,1]}\\
&\le (3+x) 4^{k-1} \le 4^k\,.
\end{split}
\end{equation*}
In view of \eqref{eq:gkgamma}, this finishes the proof.
\end{proof}


\noindent{\bf Acknowledgments}
A large part of this paper was written during a research visit of JH at the Statistics Department of 
Columbia University in the City of New York. JH thanks the Department for its great hospitality.
JH and TM  thank Richard A. Davis, Olivier Wintenberger and Gennady Samorodnitsky for valuable discussions on the subject.
\bibliography{libraryjohannes}

\begin{thebibliography}{10}

\bibitem{adamczak:litvak:pajor:jaegermann:2010}
{\sc Adamczak, R., Litvak, A.~E., Pajor, A., and Tomczak-Jaegermann, N.}
\newblock Quantitative estimates of the convergence of the empirical covariance
  matrix in log-concave ensembles.
\newblock {\em J. Amer. Math. Soc. 23}, 2 (2010), 535--561.

\bibitem{adamczak:litvak:pajor:jaegermann:2011}
{\sc Adamczak, R., Litvak, A.~E., Pajor, A., and Tomczak-Jaegermann, N.}
\newblock Sharp bounds on the rate of convergence of the empirical covariance
  matrix.
\newblock {\em C. R. Math. Acad. Sci. Paris 349}, 3-4 (2011), 195--200.

\bibitem{anderson:1963}
{\sc Anderson, T.~W.}
\newblock Asymptotic theory for principal component analysis.
\newblock {\em Ann. Math. Statist. 34\/} (1963), 122--148.

\bibitem{auffinger:arous:peche:2009}
{\sc Auffinger, A., Ben~Arous, G., and P{\'e}ch{\'e}, S.}
\newblock Poisson convergence for the largest eigenvalues of heavy tailed
  random matrices.
\newblock {\em Ann. Inst. Henri Poincar\'e Probab. Stat. 45}, 3 (2009),
  589--610.

\bibitem{bai:fang:liang:2014}
{\sc Bai, Z., Fang, Z., and Liang, Y.-C.}
\newblock {\em Spectral theory of large dimensional random matrices and its
  applications to wireless communications and finance statistics}.
\newblock World Scientific Publishing Co. Pte. Ltd., Hackensack, NJ; University
  of Science and Technology of China Press, Hefei, 2014.
\newblock Random matrix theory and its applications.

\bibitem{bai:silverstein:2010}
{\sc Bai, Z., and Silverstein, J.~W.}
\newblock {\em Spectral Analysis of Large Dimensional Random Matrices},
  second~ed.
\newblock Springer Series in Statistics. Springer, New York, 2010.

\bibitem{bai:zhou:2008}
{\sc Bai, Z., and Zhou, W.}
\newblock Large sample covariance matrices without independence structures in
  columns.
\newblock {\em Statist. Sinica 18}, 2 (2008), 425--442.

\bibitem{baisilv}
{\sc Bai, Z.~D., Silverstein, J.~W., and Yin, Y.~Q.}
\newblock A note on the largest eigenvalue of a large-dimensional sample
  covariance matrix.
\newblock {\em J. Multivariate Anal. 26}, 2 (1988), 166--168.

\bibitem{bai:yin:1993}
{\sc Bai, Z.~D., and Yin, Y.~Q.}
\newblock Limit of the smallest eigenvalue of a large-dimensional sample
  covariance matrix.
\newblock {\em Ann. Probab. 21}, 3 (1993), 1275--1294.

\bibitem{banna:2016}
{\sc Banna, M.}
\newblock Limiting spectral distribution of {G}ram matrices associated with
  functionals of {$\beta$}-mixing processes.
\newblock {\em J. Math. Anal. Appl. 433}, 1 (2016), 416--433.

\bibitem{banna:merlevede:2015}
{\sc Banna, M., and Merlev{\`e}de, F.}
\newblock Limiting spectral distribution of large sample covariance matrices
  associated with a class of stationary processes.
\newblock {\em J. Theoret. Probab. 28}, 2 (2015), 745--783.

\bibitem{banna:merlevede:peligrad:2015}
{\sc Banna, M., Merlev{\`e}de, F., and Peligrad, M.}
\newblock On the limiting spectral distribution for a large class of symmetric
  random matrices with correlated entries.
\newblock {\em Stochastic Process. Appl. 125}, 7 (2015), 2700--2726.

\bibitem{bao:pan:zhou:2012}
{\sc Bao, Z., Pan, G., and Zhou, W.}
\newblock Tracy-{W}idom law for the extreme eigenvalues of sample correlation
  matrices.
\newblock {\em Electron. J. Probab. 17\/} (2012), no. 88, 32.

\bibitem{bhatia:1997}
{\sc Bhatia, R.}
\newblock {\em Matrix Analysis}, vol.~169 of {\em Graduate Texts in
  Mathematics}.
\newblock Springer-Verlag, New York, 1997.

\bibitem{davis:mikosch:heiny:xie:2015}
{\sc Davis, R.~A., Heiny, J., Mikosch, T., and Xie, X.}
\newblock Extreme value analysis for the sample autocovariance matrices of
  heavy-tailed multivariate time series.
\newblock {\em Extremes 19}, 3 (2016), 517--547.

\bibitem{davis:mikosch:pfaffel:2015}
{\sc Davis, R.~A., Mikosch, T., and Pfaffel, O.}
\newblock Asymptotic theory for the sample covariance matrix of a heavy-tailed
  multivariate time series.
\newblock {\em Stochastic Process. Appl. 126}, 3 (2016), 767--799.

\bibitem{davis:pfaffel:stelzer:2014}
{\sc Davis, R.~A., Pfaffel, O., and Stelzer, R.}
\newblock Limit theory for the largest eigenvalues of sample covariance
  matrices with heavy-tails.
\newblock {\em Stochastic Process. Appl. 124}, 1 (2014), 18--50.

\bibitem{dobriban:2015}
{\sc Dobriban, E.}
\newblock Efficient computation of limit spectra of sample covariance matrices.
\newblock {\em Random Matrices Theory Appl. 4}, 4 (2015), 1550019, 36.

\bibitem{donoho:2000}
{\sc Donoho, D.}
\newblock High-dimensional data analysis: the curses and blessings of
  dimensionality.
\newblock {\em Technical Report, Stanford University\/} (2000).

\bibitem{elkaroui:2008}
{\sc El~Karoui, N.}
\newblock Spectrum estimation for large dimensional covariance matrices using
  random matrix theory.
\newblock {\em Ann. Statist. 36}, 6 (2008), 2757--2790.

\bibitem{elkaroui:2009}
{\sc El~Karoui, N.}
\newblock Concentration of measure and spectra of random matrices: applications
  to correlation matrices, elliptical distributions and beyond.
\newblock {\em Ann. Appl. Probab. 19}, 6 (2009), 2362--2405.

\bibitem{elkaroui:purdom:2016}
{\sc El~Karoui, N., and Purdom, E.}
\newblock The bootstrap, covariance matrices and pca in moderate and
  high-dimensions.
\newblock {\em Technical Report, University of California\/} (2016).

\bibitem{geman}
{\sc Geman, S.}
\newblock A limit theorem for the norm of random matrices.
\newblock {\em Ann. Probab. 8}, 2 (1980), 252--261.

\bibitem{gine:goetze:mason:1997}
{\sc Gin{\'e}, E., G{\"o}tze, F., and Mason, D.~M.}
\newblock When is the {S}tudent {$t$}-statistic asymptotically standard normal?
\newblock {\em Ann. Probab. 25}, 3 (1997), 1514--1531.

\bibitem{heiny:mikosch:2015:iid}
{\sc Heiny, J., and Mikosch, T.}
\newblock Eigenvalues and eigenvectors of heavy-tailed sample covariance
  matrices with general growth rates: the iid case.
\newblock {\em Stochastic Process. Appl.\/} (2016), 29.

\bibitem{janssen:mikosch:rezapour:xie:2016}
{\sc Janssen, A., Mikosch, T., Mohsen, R., and Xiaolei, X.}
\newblock The eigenvalues of the sample covariance matrix of a multivariate
  heavy-tailed stochastic volatility model.
\newblock {\em Available at {\tt https://arxiv.org/abs/1605.02563}\/} (2016).

\bibitem{jiang:2004}
{\sc Jiang, T.}
\newblock The limiting distributions of eigenvalues of sample correlation
  matrices.
\newblock {\em Sankhy\=a 66}, 1 (2004), 35--48.

\bibitem{johnstone:2001}
{\sc Johnstone, I.~M.}
\newblock On the distribution of the largest eigenvalue in principal components
  analysis.
\newblock {\em Ann. Statist. 29}, 2 (2001), 295--327.

\bibitem{jonsson:2010}
{\sc Jonsson, F.}
\newblock On the quadratic moment of self-normalized sums.
\newblock {\em Statist. Probab. Lett. 80}, 17-18 (2010), 1289--1296.

\bibitem{marchenko:pastur:1967}
{\sc Mar{\v{c}}enko, V.~A., and Pastur, L.~A.}
\newblock Distribution of eigenvalues in certain sets of random matrices.
\newblock {\em Mat. Sb. (N.S.) 72 (114)\/} (1967), 507--536.

\bibitem{mason:zinn:2005}
{\sc Mason, D.~M., and Zinn, J.}
\newblock Acknowledgment of priority: ``{W}hen does a randomly weighted
  self-normalized sum converge in distribution?'' [{E}lectron. {C}omm.
  {P}robab. {\bf 10} (2005), 70--81 (electronic); mr2133894].
\newblock {\em Electron. Comm. Probab. 10\/} (2005), 297 (electronic).

\bibitem{muirhead}
{\sc Muirhead, R.~J.}
\newblock {\em Aspects of Multivariate Statistical Theory}.
\newblock John Wiley \& Sons, Inc., New York, 1982.
\newblock Wiley Series in Probability and Mathematical Statistics.

\bibitem{soshnikov:2004}
{\sc Soshnikov, A.}
\newblock Poisson statistics for the largest eigenvalues of {W}igner random
  matrices with heavy tails.
\newblock {\em Electron. Comm. Probab. 9\/} (2004), 82--91 (electronic).

\bibitem{soshnikov:2006}
{\sc Soshnikov, A.}
\newblock Poisson statistics for the largest eigenvalues in random matrix
  ensembles.
\newblock In {\em Mathematical physics of quantum mechanics}, vol.~690 of {\em
  Lecture Notes in Phys.} Springer, Berlin, 2006, pp.~351--364.

\bibitem{srivastava:vershynin:2013}
{\sc Srivastava, N., and Vershynin, R.}
\newblock Covariance estimation for distributions with {$2+\varepsilon$}
  moments.
\newblock {\em Ann. Probab. 41}, 5 (2013), 3081--3111.

\bibitem{tikhomirov:2015}
{\sc Tikhomirov, K.}
\newblock The limit of the smallest singular value of random matrices with
  i.i.d. entries.
\newblock {\em Adv. Math. 284\/} (2015), 1--20.

\bibitem{tracy:widom:2012}
{\sc Tracy, C.~A., and Widom, H.}
\newblock Distribution functions for largest eigenvalues and their
  applications.
\newblock In {\em Proceedings of the {I}nternational {C}ongress of
  {M}athematicians, {V}ol. {I} ({B}eijing, 2002)\/} (2002), Higher Ed. Press,
  Beijing, pp.~587--596.

\bibitem{xiao:zhou:2010}
{\sc Xiao, H., and Zhou, W.}
\newblock Almost sure limit of the smallest eigenvalue of some sample
  correlation matrices.
\newblock {\em J. Theoret. Probab. 23}, 1 (2010), 1--20.

\bibitem{yao:zheng:bai:2015}
{\sc Yao, J., Zheng, S., and Bai, Z.}
\newblock {\em Large sample covariance matrices and high-dimensional data
  analysis}.
\newblock Cambridge Series in Statistical and Probabilistic Mathematics.
  Cambridge University Press, New York, 2015.

\bibitem{bai:yin:krishnaiah:1988}
{\sc Yin, Y.~Q., Bai, Z.~D., and Krishnaiah, P.~R.}
\newblock On the limit of the largest eigenvalue of the large-dimensional
  sample covariance matrix.
\newblock {\em Probab. Theory Related Fields 78}, 4 (1988), 509--521.

\end{thebibliography}


\end{document}